\documentclass[12pt]{amsart}
\usepackage{amsthm,amsmath,amssymb}
\usepackage{graphicx}
\usepackage[colorlinks=true,citecolor=black,linkcolor=black,urlcolor=blue]{hyperref}
\usepackage{tikz-cd}
\usepackage{algpseudocode} 
\usepackage{tikz}
\usetikzlibrary{fit,positioning,hobby,calc}

\newcommand{\R}{\mathbb{R}}
\newcommand{\FF}{\mathbb{F}}

\newcommand{\e}{\varepsilon}
\newcommand{\ir}{\mathord{\mathrm{ir}}}

\newcommand{\di}{\partial}
\renewcommand{\a}{\alpha}
\newcommand{\f}{\varphi}

\newcommand{\hide}[1]{}

\def\:{\mkern 1.2mu \colon}
\newcommand{\us}{\mathord{\mathrm{Upper}}}
\newcommand{\ls}{\mathord{\mathrm{Lower}}}

\definecolor{rdeca}{rgb}{0.81, 0.09, 0.13}
\definecolor{oranzna}{rgb}{1.0, 0.45, 0.0}
\definecolor{rumena}{rgb}{1.0, 0.70, 0.0}
\definecolor{szelena}{rgb}{0.56, 0.62, 0.06}
\definecolor{zelena}{rgb}{0.13, 0.55, 0.13}
\definecolor{smodra}{rgb}{0.0, 0.65, 0.83}
\definecolor{modra}{rgb}{0.0, 0.25, 0.7}
\definecolor{sliva}{rgb}{0.56, 0.27, 0.52}
\definecolor{mz}{rgb}{0.065, 0.385, 0.395}
\definecolor{mv}{rgb}{0.28, 0.245, 0.59}

\theoremstyle{plain}
\newtheorem{theorem}{Theorem}
\newtheorem{lemma}[theorem]{Lemma}
\newtheorem{corollary}[theorem]{Corollary}
\newtheorem{proposition}[theorem]{Proposition}

\theoremstyle{definition}
\newtheorem{definition}[theorem]{Definition}
\newtheorem{example}[theorem]{Example}

\theoremstyle{remark}
\newtheorem{remark}[theorem]{Remark}

\begin{document}
\title{Rigidity of terminal simplices in persistent homology}

\author{Aleksandra Franc}
\address{University of Ljubljana, Slovenia}
\email{aleksandra.franc@fri.uni-lj.si}

\author{\v{Z}iga Virk}
\address{University of Ljubljana, Slovenia, and Institute IMFM, Ljubljana, Slovenia}
\email{ziga.virk@fri.uni-lj.si}

\thanks{The authors were supported by Slovenian Research Agency grant No. N1-0114. The second named author was also supported by Slovenian Research Agency grants No. J1-4001, J1-4031, and P1-0292.}\

\maketitle
 
\renewcommand{\thefootnote}{\fnsymbol{footnote}} 
\footnotetext{The authors would like to thank the referee for a careful reading and valuable comments.}
\renewcommand{\thefootnote}{\arabic{footnote}}

\begin{abstract}
Given a filtration function on a finite simplicial complex, stability theorem of persistent homology states that the corresponding barcode is continuous with respect to changes in the filtration function. However, due to the discrete setting of simplicial complexes, the simplices terminating matched bars cannot change continuously for arbitrary perturbations of filtration functions. In this paper we provide a sufficient condition for rigidity of a terminal simplex, i.e., a condition on $\e>0$ implying that the terminal simplex of a homology class or a bar in persistent homology remains constant through $\e$-perturbations of filtration function. The condition for a homology class or a bar in dimension $n$ depends only on the barcodes in dimensions $n$ and $n+1$.
\end{abstract}

\section{Introduction}

Let $K$ be a finite simplicial complex. A \textbf{filtration function} assigns to each simplex in $K$ a unique real value, such that for each simplex, the values of its faces are lower than the value of the simplex. For each $r\in \R$ we define $K_r$ as the subcomplex of $K$ consisting of all the simplices whose values are at most $r$. The collection of subcomplexes $\{K_r\}_{r\in \R}$ connected by the natural inclusions $K_r \to K_{r'}$ for each $r \leq r'$  is called a filtration of $K$. 

Persistent homology \cite{EdelsZomo, EdelsHarer} is a parameterized version of homology, with the parameter arising from a filtration function of a simplicial complex. It is obtained by applying a homology to a filtration, which yields a collection of homology groups $\{H_n(K_r)\}_{r\in\R}$ and inclusion-induced maps between them. One of its fundamental features as compared to homology is stability \cite{EdelsMoro}. In particular, small perturbations of filtration functions induce small perturbations to the lifespans (along parameter $r$) of homology classes. However, the terminal simplices and the corresponding homology representatives of persistent homology cannot change continuously with arbitrary perturbations of filtration functions. 

As simplices keep appearing in a filtration, non-trivial homology classes are either appearing (if the boundary of the added simplex is homologically trivial) or terminating. A simplex terminating a homology class (i.e., identifying the class with the trivial class) is called a \textbf{terminal} simplex. 
A corresponding homology representative, appearing at the first possible scale of the filtration, can be obtained from the boundary of the terminal simplex. As such, the terminal simplices allow us to localize and determine a geometric manifestation of homology terminating with the simplex. While finding a suitable homology representative in persistent homology is by itself challenging \cite{Chao}, the fact that small perturbations of a filtration function may yield completely different terminal simplices results in unstable representatives. There has been an attempt to circumvent this issue in practice \cite{Bub}.

In this paper we study the region of parameter $\e$ for which the terminal  simplex of a persistent homology class $[\alpha]$ of an injective  filtration function $f$  is \textbf{rigid} (i.e., constant) through $\e$-perturbations of $f$. Let $[a,b)$ be the lifespan interval of an $n$-dimensional homology class $[\alpha].$ Our \textbf{main results} are the following (under suitable assumptions).
\begin{enumerate}
\item Theorem \ref{ThmConseq1}: The terminal $(n+1)$-dimensional simplex of $[\alpha]$ is rigid for $\e$-perturbations of $f$ if
\begin{itemize}
\item no class of $H_{n+1}$ is born on $(b,b+2\e]$ and 
\item no class of $H_n$ terminates on $[b-2\e,b)$.
\end{itemize}
\item Theorem \ref{ThmBarcodeFinal}: A version of Theorem \ref{ThmConseq1} for significant bars in the barcode. One of the main \textbf{advantages} of this result is that rigidity can be deduced solely from the barcode, without looking at the underlying filtration. (For details on barcodes of persistent homology see Preliminaries.)

\end{enumerate}

Theorem \ref{ThmConseq1} describes the two ways in which non-rigidity of a terminal simplex may occur: either via sequentially critical pair of simplices or independently critical pair of simplices. As a result, we are not only able to provide bounds on the region of rigidity, but also to locate a simplex ($\Delta_2$ in Theorem \ref{ThmConseq1}) appearing as a new terminal simplex for $[\alpha]$ in the region of non-rigidity. For a demonstration within the context of persistence diagrams and barcodes (defined in Preliminaries) see Figure \ref{FigMainIntro1g} and Figure \ref{FigMainIntro2g} following Theorem \ref{ThmBarcodeFinal}.

On the other hand, our results contribute to a new interpretation of the structure of persistent homology. So far, persistent homology has been known to encode topological information of the space at small scales \cite{Haus, Lat, ZV3}, intrinsic combinatorial structure of filtrations (such as Rips complexes) \cite{AA, Ad5}, proximity of spaces via the stability result \cite{ZVCounterex}, geometric properties of spaces  \cite{ZV2}, shortest homology basis \cite{ZV}, spaces of contraction \cite{ZVCont}, filling radius \cite{Memoli}, curvature \cite{Bubenik2020}, width of homology class \cite{ACos}, and more. Our results imply that,  to a degree, persistent homology encodes rigidity of terminal simplices.

\section{Preliminaries}


We first introduce the setup of persistent homology, see \cite{EdelsHarer} for details. Throughout this paper we assume that $K$ is a finite simplicial complex and $f\colon K \to \mathbb{R}$ is an injective \textbf{filtration function} on $K$ (if $\sigma$ is a face of $\tau$ then $f(\sigma)< f(\tau)$ for all $\sigma, \tau\in K$). As such $f$ encodes an \textbf{order} on the simplices of $K$. For example, in the original persistent homology algorithm \cite{EdelsZomo} such an order is used to arrange simplices in the boundary matrix. A simplex in $K$ is (inclusion) \textbf{maximal} if it is not a proper face of any simplex.

Given two injective filtration functions $f$ and $g$, we define the \textbf{distance} between them as
$$||f-g||_\infty = \max_{\sigma\in K}|f(\sigma)-g(\sigma)|.$$
Function $g$ corresponds to some permutation of the ordering of simplices encoded by $f$. 

The \textbf{sublevel sets} of $f$ are subcomplexes of $K$ defined for all $r\in\mathbb{R}$ as the pre-images $K_r^f = f^{-1}((-\infty, r])$. We can also define $K^f_\infty = f^{-1}((-\infty, \infty)) = K$. The notation $K^f$ denotes the standard \textbf{sublevel filtration} of $K$ obtained through $f$, i.e., the collection of subcomplexes $\{K^f_r\}_{r\in \R}$ along with the natural inclusions 
$$
\iota_{q,r} \colon K^f_q \to K^f_r
$$
for all $q\leq r$. Applying homology $H_n$ as a functor to a filtration we obtain a \textbf{persistence module}, i.e., a collection of vector spaces $\{H_n(K^f_r)\}_{r\in \R}$ along with the inclusions induced linear maps
$$\iota^*_{q,r} \colon H_n(K^f_q) \to H_n(K^f_r)$$
for all $q\leq r$.
All homology groups are assumed to be with coefficients in a fixed field $\FF$ and therefore not mentioned in the notation for homology.

Given a non-trivial homology element $[\alpha]\in H_n(K^f_r)$ for some $r$, we define:
\begin{description}
\item [birth] $a\in \R$ of $[\alpha]$ as the infimum of levels $q \leq r$, for which there exists $[\alpha_q]\in H_n(K^f_q)$ such that $\iota^*_{q,r} [\alpha_q]=[\alpha]$. We say that $[\alpha]$ is born at $a$.
\item [termination scale] $b\in \R \cup \{\infty\}$ of $[\alpha]$ as the infimum of levels $q \geq r$, for which $\iota^*_{r,q} [\alpha]=0$, or $\infty$ if such levels do not exist. We say that $[\alpha]$ is terminated at $b$. 
\end{description}
Note that the termination scale is defined differently than death in \cite{EdelsHarer}.
Given our setup of sublevel complexes defined through preimages of closed intervals, the infima in the definition of birth and termination scale are always attained if finite. 
Since $f$ is injective, at most one simplex is added at each level $r$. Consequently each $(n+1)$-simplex either gives birth to a non-trivial homology class in dimension $n+1$ or terminates a non-trivial homology class in dimension $n$.

We next state the stability theorem and introduce the corresponding notation, see \cite{Bauer} for details.
Each persistence module obtained in our setting decomposes as a finite direct sum of \textbf{interval modules} $\FF_{[a_i,b_i)}$, where persistence module $\FF_{[a,b)}$ for $a<b$ is a collection of vector spaces $\{V_r\}_{r\in \R}$, with:
\begin{itemize}
\item  $V_r=\FF$ for $r\in [a,b)$,
\item  $V_r=0$ for $r\notin [a,b)$ and
\item the bonding linear maps $V_r \to V_{r'}$ being identities for parameters $r<r'$ from $[a,b)$.
\end{itemize}
Note that the type of endpoints of intervals (closed on the left, open on the right) is a consequence of our setup of a filtration, i.e., a sublevel filtration of an injective filtration function on a finite simplicial complex. The collection of intervals $[a_i,b_i)$ is called a \textbf{barcode} and a single interval in this setting is referred to as a \textbf{bar}. 

Given injective filtration functions $f$ and $g$ and $n\in \{0,1,\ldots\}$, assume 
$$
\mathcal{M}=\{H_n(K^f_r)\}_{r\in \R} = \bigoplus_{i\in I} \FF_{[a_i,b_i)}
$$
and
$$
\mathcal{M}'=\{H_n(K^g_r)\}_{r\in \R} = \bigoplus_{j\in J} \FF_{[a'_j,b'_j)}
$$
are decompositions of persistence modules $\mathcal{M}$ and $\mathcal{M}'$ into interval modules. The \textbf{bottleneck distance} between $\mathcal{M}$ and $\mathcal{M}'$ is the infimum of $\e>0$, for which there exists a bijection $\f \colon I' \to J'$ for subsets $I' \subseteq I$ and $J' \subseteq J$ such that:
\begin{itemize}
\item $|a_i - a'_{\f(i)}| \leq \e$ and $|b_i - b'_{\f(i)}| \leq \e$ for all $i\in I'$,
\item $|a_i-b_i | \leq 2\e$ for all $i\in I\setminus I'$ and
\item $|a'_j-b'_j | \leq 2\e$ for all $j\in J\setminus J'$.
\end{itemize}
The stability theorem states that if $||f-g||_\infty \leq \e$, then the bottleneck distance between $\mathcal{M}$ and $\mathcal{M}'$ is at most $\e$, see \cite{EdelsHarer} for details.

\subsection{Filtration manipulation}
\label{SubsFM}

The following propositions explain local adjustments to filtration functions that result in a predetermined permutation of a collection of simplices.

\begin{proposition}
\label{switch}
Let $\sigma_1$ and $\sigma_2$ be two $n$-dimensional simplices in a finite simplicial complex $K$ and let $f$ be an injective filtration function on $K$. Assume that for some $\e > 0$ we have
$$f(\sigma_1) < f(\sigma_2)< f(\sigma_1)+2\e.$$
Then there exists an injective filtration function $g$ on $K$ such that $||f-g||_\infty \leq \e$ and $g(\sigma_2) < g(\sigma_1)$.
\end{proposition}

\begin{proof}
Without loss of generality we can slightly decrease $\e$ so that the assumptions of the proposition still hold and that
\begin{equation}
\label{injectiveg}
\e \notin \{|f(\sigma)- f(\tau)| \;;\; \sigma, \tau \in K\} \textrm{ and }2\e \notin \{|f(\sigma)- f(\tau)| \;;\; \sigma, \tau \in K\}.
\end{equation}

Let $U=\us(\sigma_1)\subseteq K$ denote the subset of all simplices in $K$ that contain $\sigma_1$ as a face (the upper set of $\sigma_1$ in the Hasse diagram of $K$), and let $L=\ls(\sigma_2)\subseteq K$ denote the subset of all faces of $\sigma_2$ (the lower set of $\sigma_2$ in the Hasse diagram of $K$). Note that $\sigma_1\in U$ and $\sigma_2\in L$.

If $\mu\in U$, set $g(\mu) = f(\mu)+\e$. If $\mu\in L$, set $g(\mu) = f(\mu)-\e$. For all other simplices $\mu$ let $g(\mu) = f(\mu)$.

Since $\sigma_1$ and $\sigma_2$ are two distinct simplices of the same dimension, we have $U\cap L = \emptyset$, so $g$ is well-defined. It is also obvious that $|g(\mu)-f(\mu)|\leq\e$ for all $\mu\in K$, so $||g-f||_\infty\leq\e$. Finally, if the value at a simplex $\mu$ has decreased (or increased) by $\e$, then the same was true for all the faces (or cofaces) of $\mu$, so $g$ is a filtration function.
The function $g$ defined this way is injective by the condition (\ref{injectiveg}).
\end{proof}

Given a collection of $k$ simplices of dimension $n$ and an arbitrary permutation $\pi\in S_k$, we do not have to switch the order of one pair at a time but can instead mix them up all at once. Corollary \ref{switchmany} shows one way of doing this.

\begin{corollary}
\label{switchmany}
Let $\sigma_1,\ldots,\sigma_k$ be a selection of $n$-dimensional simplices in a finite simplicial complex $K$, $\pi\in S_k$ an arbitrary permutation of the indices $\{1,2,\ldots,k\}$ and $f$ an injective filtration function on $K$. Assume that for some $\e > 0$ we have
$$f(\sigma_1) < f(\sigma_2)< \ldots < f(\sigma_k) < f(\sigma_1)+2\e.$$
Then there exists an injective filtration function $g$ on $K$ such that $||f-g||_\infty \leq \e$ and
$$g(\sigma_{\pi(1)}) < g(\sigma_{\pi(2)})< \ldots < g(\sigma_{\pi(k)}).$$
\end{corollary}
\begin{proof}
If we write $a = f(\sigma_k)-\e$ and $b = f(\sigma_1)+\e$, then $(a,b\,]$ is an interval of length $\delta < 2\e$. Divide this interval into $k+1$ equal pieces of length $\frac{\delta}{k+1}$ and define
\begin{eqnarray*}
g(\sigma_{\pi(1)}) &=& a +\frac{\delta}{k+1},\\
g(\sigma_{\pi(2)}) &=& a +\frac{2\delta}{k+1},\\
 &\vdots & \\
g(\sigma_{\pi(k-1)}) &=& a +\frac{(k-1)\delta}{k+1},\\
g(\sigma_{\pi(k)}) &=& a +\frac{k\delta}{k+1}.
\end{eqnarray*}
It is not difficult to show that
$|g(\sigma_{i})-f(\sigma_{i})| \leq \e - \frac{\delta}{k+1}$.
Define 
\begin{eqnarray*}
I^+ &=& \{i\in\{1,2,\ldots,k\}\;;\; g(\sigma_i)-f(\sigma_i) > 0\},\\
I^- &=& \{i\in\{1,2,\ldots,k\}\;;\; g(\sigma_i)-f(\sigma_i) < 0\},\\
\e^+ &=& \max\{g(\sigma_i)-f(\sigma_i)\;;\; i\in I^+\},\\
\e^- &=& \max\{|g(\sigma_i)-f(\sigma_i)|\;;\; i\in I^-\},\\
\end{eqnarray*}
\vskip-1.2cm
$$U = \bigcup_{i\in I^+}\us(\sigma_i)\quad\textrm{and}\quad L = \bigcup_{i\in I^-}\ls(\sigma_i).$$
It is easy to see that $\e^+ <\e$ and $\e^- <\e$.
For all $\mu\in U\setminus\{\sigma_1,\ldots,\sigma_k\}$ let $g(\mu) = f(\mu)+\e^+$. For all $\mu\in L\setminus\{\sigma_1,\ldots,\sigma_k\}$ let $g(\mu) = f(\mu)-\e^-$. Finally, for all other $\mu\in K$ set $g(\mu) = f(\mu)$. Then $g$ is a filtration function with all the desired properties except (perhaps) injectivity, but it can be made injective either with a small perturbation of values (upholding the property of being a filtration function), or with a minor decrease in $\e$ at the beginning (as was done in the proof of Proposition \ref{switch}).
\end{proof}

\section{Rigidity for homology classes}

As before, let $f$ be an injective filtration function on a simplicial complex $K$ and let $n$ be a positive integer. Assume that an $n$-cycle $\alpha$ is created when its last $n$-dimensional simplex is added at level $a$, and that the homology class $[\a]\in H_n(K_a^f)$ is born at $a$ and terminated at $b$. A \textbf{nullhomology} of $\alpha$ at $b$ is an $(n+1)$-chain in $K_b^f$, whose boundary is $\alpha$.

Choose $\e<\frac{b-a}{2}$. Then for each $\e$-perturbation $g$ of $f$ the class $[\a]$ exists and is non-trivial in $H_n(K_r^g)$ at least for $r\in [a+\e, b-\e)$.

For every injective filtration function $g$ at distance at most $\e<\frac{b-a}{2}$ from $f$ let $\Delta_{g,\a}$ denote the $(n+1)$-simplex in $K^g$ that terminates $[\a]$. Define
$$\Sigma_{\e} = \{\Delta_{g,\a}\;;\; \textrm{$g$ an injective $\e$-perturbation of $f$}\}.$$
Note that $g(\Delta_{g,\a})\in [b-\e, b+\e]$ and $f(\Delta_{g,\a})\in [b-2\e,b+2\e]$. We say that $[\alpha]$ is  $\e$-\textbf{terminally-rigid} if $|\Sigma_{\e}|=1$.

Since $f$ is injective and $K$ is finite, we can define
$$\ir(f) = \min\{|f(\sigma)-f(\tau)|\;;\;\sigma \neq \tau\in K\}>0,$$
the \textbf{injectivity radius} of $f$.

The aim of this paper is to study algebraic effects of terminal non-rigidity. Homology class $[\alpha]$ as defined above is $\e$-terminally-rigid for small $\e$, certainly for $\e<\frac{\ir(f)}{2}$ since $\frac{\ir(f)}{2}$-perturbations of $f$ retain the order of simplices appearing in the filtration of $K$. For larger $\e$ the class $[\alpha]$ is typically not terminally-rigid. We intend to focus on the region of $\e$ in which the initial form of non-rigidity occurs.

Let $d = \dim(K)$ and let $\sigma_j^{(i)}$ for $j=1,\ldots,k_i$ be the simplices of $K$ of dimension $i$. Then $f$ determines a linear ordering on the set
$$\{\sigma_1^{(0)},\ldots,\sigma_{k_0}^{(0)},\ldots,\sigma_1^{(d)},\ldots,\sigma_{k_{d}}^{(d)}\}$$
of all simplices of $K$. We can encode this ordering as a permutation $\overline{\pi}^f\in S_{k_0+k_1+\ldots+k_d}$ of the pre-determined ordering of simplices given above. A different injective filtration function $h$ defines a potentially different ordering of the simplices of $K$, corresponding to a potentially different permutation $\overline{\pi}^h$. Note that not all permutations in the symmetric group $S_{k_0+k_1+\ldots+k_d}$ correspond to filtration functions.

If we limit ourselves to only the simplices of a single dimension, however, all possible permutations of those simplices can be realized by Corollary \ref{switchmany}. Our filtration function $f$ determines a permutation $\pi_{n+1}^f\in S_{k_{n+1}}$, defined by the induced linear ordering of the set
$$\{\sigma_1^{(n+1)},\ldots,\sigma_{k_{n+1}}^{(n+1)}\}$$
of $(n+1)$-dimensional simplices of $K$. We will assume that the simplices of $K$ have been ordered in such a way that $\pi_{n+1}^f$ is the identity permutation. To unburden the notation we will write $\pi_{n+1}^f = \pi^f$ and $m=k_{n+1}$ from now on.

Any injective filtration functions that correspond to the same permutation generate the same boundary matrices in the classical matrix reduction algorithm for persistent homology, although the labels (function values) of simplices generally differ.

Also recall that for each $\e\in D=(0,\frac{b-a}{2}]$ and each injective filtration function $g$ with $||f-g||_\infty \leq \e$ the class $[\alpha]$ has a non-trivial lifespan in $K^g$.

Now, let us consider
\begin{itemize}
\item the function $\e \mapsto \Sigma_\e$ defined on $D$ that returns the collection of all terminal simplices of $\e$-perturbations of $f$ that terminate $[\alpha]$ and
\item the function $\e \mapsto \Pi_\e$ defined on $D$ that returns the collection of all permutations corresponding to $\e$-perturbations of $f$.
\end{itemize}

The functions $\e\mapsto|\Sigma_\e|$ and $\e\mapsto|\Pi_\e|$ defined on $D$
\begin{enumerate}
\item are increasing,
\item have values in the discrete sets $\{1,2,\ldots, m\}$ and $\{1,2,\ldots, m!\}$, respectively, and
\item attain the value of $1$ for small $\e$ (at least for all $\e<\frac{\ir(f)}{2}$).
\end{enumerate}

As a result, both functions are increasing step functions on $D$ as each of them partitions $D$ into finitely many intervals, such that the function is constant on each interval of that partition. It is also apparent that a change in $\Sigma_\e$ may occur at some $\e$ only if $\Pi_\e$ also changes at the same parameter. The next lemma shows that $|\Sigma_\e|$ and $|\Pi_\e|$ are lower semi-continuous (the intervals of the two partitions of $D$ are open on the left and closed on the right).

\begin{lemma}
\label{LemmaUSC}
Assume $f$ is an injective filtration function on a finite simplicial complex $K$ and assume that for the $n$-cycle $\alpha$ appearing in the filtration at level $a$ the homology class $[\alpha]\in H_n(K^f_a)$ is born at $a$ and terminated at $b$. For each $t\in D = \left(0,\frac{b-a}{2}\right]$ there exists $\delta>0$ such that $\Sigma_\e$  and $\Pi_\e$ are constant on $(t-\delta,t]$.
\end{lemma}

\begin{proof}
Choose finitely many injective filtration functions generating $\Pi_t$. Each of them can be brought a bit closer to $f$ because of injectivity.
\end{proof}

We next discuss potential points of discontinuity of $|\Sigma_\e|$ and $|\Pi_\e|$. The following lemma shows that they always correspond to exactly half the distance between two (not necessarily consecutive) values of $f$.

\begin{lemma}
\label{LemmaDiscontinuity}
Assume $f$ is an injective filtration function on a finite simplicial complex $K$ and assume that for the $n$-cycle $\alpha$ appearing in the filtration at level $a$ the homology class $[\alpha]\in H_n(K^f_a)$ is born at $a$ and terminated at $b$. Suppose $t\in D\left(0,\frac{b-a}{2}\right]$ is a point of discontinuity of $\Pi_\e$. Then there exist $i\neq j$ such that $2t=|f(\sigma_i)-f(\sigma_j)|$. 
\end{lemma}

\begin{proof}
As a consequence of Lemma \ref{LemmaUSC} there exists a $\delta>0$ such that $\Pi_\e$ is constant on $(t,t+\delta]$. Since $t$ is the point of discontinuity, there exists a permutation $\nu \in \Pi_{t+\delta}\setminus \Pi_t$. For each $N\in \mathbb{N}$ choose an injective filtration function $g_N$ inducing $\nu$ and satisfying $||f-g_N||_\infty \leq t+ 1/N$. Without loss of generality we may assume that $\{g_N(\sigma_i)\}_{N\in \mathbb{N}}$ converges for each $i$ and define $f_\infty$ as the limiting function (if any of them do not converge we can choose a converging subsequence). Note that while $f_\infty$ is a filtration function, it cannot be injective, because it would have corresponded to the premutation $\nu$ which is not in $\Pi_t$ while $||f-f_\infty||_\infty \leq t$ by definition.
Let us try fixing the values of $f_\infty$ to make it injective and see where exactly that fails.

Assume that $q$ is a value attained by $f_\infty$ at more than one simplex. Let $A_q$ denote the collection of such simplices. If there exist two distinct $\sigma_i, \sigma_j$ such that $q = f(\sigma_i)-t = f(\sigma_j)+t$, then we have found two simplices that satisfy the desired condition. If not, there are two possible reasons for that.
\begin{itemize}
\item If there is no $i$ such that $q=f(\sigma_i)-t$, then the $f$-values of all simplices in $A_q$ lie on $[q-t,q+t)$. In this case the $f_\infty$ values of all the simplices in $A_q$ can be slightly decreased by Corollary \ref{switchmany} so that:
\begin{itemize}
\item $f_\infty(\sigma)\neq f_\infty(\sigma')$ for all $\sigma\neq\sigma'$ in $A_q$,
\item the relative position of each simplex of $A_q$ to the simplices outside of $A_q$ is the same as in $\nu$,
\item simplices of $A_q$ appear in the same order as in $\nu$ and
\item the resulting filtration function $f'_\infty$ satisfies $||f-f'_\infty||_\infty \leq t$. 
\end{itemize}
Since this would make $f_\infty$ injective, it cannot happen for all $q$ and there must be a simplex $\sigma_i$ such that $q = f(\sigma_i)-t$.
\item The case where there is no $j$ such that $q = f(\sigma_j)+t$ and the $f$-values of simplices from $A_q$ all lie on $(q-t,q+t]$ can be handled similarly by a small local increase of the values of $f_\infty$.
\end{itemize}
We can conclude that there exist $i\neq j$ such that $q = f(\sigma_i)-t = f(\sigma_j)+t$.
\end{proof}

\begin{definition}
An injective filtration function $f$ is called \textbf{generic} if 
$$|f(\sigma_i)-f(\sigma_j)|=|f(\sigma_{i'})-f(\sigma_{j'})|$$ 
implies
$\{i,j\}=\{i',j'\}$.
\end{definition}

\begin{proposition}
\label{PropStep2}
Let $f$ be a generic injective filtration function on a finite simplicial complex $K$. Assume that for the $n$-cycle $\alpha$ appearing in the filtration at level $a$ the homology class $[\alpha]\in H_n(K^f_a)$ is born at $a$ and terminated at $b$ and let $D\left(0,\frac{b-a}{2}\right]$. Let $t_0=\max\{x\in D; |\Sigma_x|=1\}$ and $\Sigma_{t_0}=\{\Delta_1\}$. By Lemma \ref{LemmaDiscontinuity} there exist indices $i,j$ satisfying $q=f(\sigma_{j})-t_0=f(\sigma_{i})+t_0$.
Choose $t>t_0$ such that $|\Sigma_{t}|$ is as small as possible, i.e., $|\Sigma_{t}|=\lim_{x \searrow t_0} |\Sigma_x|$.
Then:
\begin{enumerate}
 \item[(1)] $\Delta_1\in \{\sigma_{i}, \sigma_{j}\}$.
 \item[(2)] $\Delta_1$ is a maximal simplex in $K^f_{f(\Delta_1)+2s}$ for all $s<t_0$.
 \item[(3)] $\Delta_1=\sigma_{j}$ implies $\{\sigma_{i}, \sigma_{j}\} \subseteq \Sigma_{t}$ and $\sigma_{i}$ terminates a class in $H_*(K^f)$. 
 \item[(4)] $\Delta_1=\sigma_{i}$ implies $\{\sigma_{i}, \sigma_{j}\} \subseteq \Sigma_{t}$ and $\sigma_{j}$ creates a class in $H_*(K^f)$. 
 \end{enumerate}
\end{proposition}

\begin{definition}
 In case (3) of Proposition \ref{PropStep2} (i.e., when $\Delta_1=\sigma_j$)  we say simplices $\sigma_i$ and $\sigma_j$ are \textbf{sequentially critical}, see Figure \ref{example5figure2}.
 
 In case (4) of Proposition \ref{PropStep2} (i.e., when $\Delta_1=\sigma_i$) we say simplices $\sigma_i$ and $\sigma_j$ are \textbf{independently critical}, see Figure \ref{example5figure1}.

\end{definition}

\begin{proof}
Without loss of generality we may choose $t$ so that no element of the form $\frac{1}{2}|f(\sigma_i)-f(\sigma_j)|$ lies on $(t_0,t]$. 
We will further develop the limiting argument presented in Lemma \ref{LemmaDiscontinuity}. Choose $\nu \in \Pi_{t}\setminus \Pi_{t_0}$. For each $N\in \mathbb{N}$ choose a filtration function $g_N$ inducing $\nu$ and satisfying $||f-g_N||_\infty \leq t_0+1/N$. Without loss of generality we may assume $\{g_N(\sigma_i)\}_{N\in \mathbb{N}}$ converges for each $i$ and define $f_\infty$ as the limiting function. Using local modifications as in the proof of Lemma \ref{LemmaDiscontinuity} we may assume that $f_\infty^{-1}(x)$ contains at most one simplex for each $x \neq q$. Define $A_q=f_\infty^{-1}(q)$. Choose an open interval $H$ around $q$ such that
$$H \subseteq \left(\bigcap_{\sigma\in \{\sigma_i, \sigma_j\}} (f(\sigma)-t,f(\sigma)+t)\right)\ \cap \  \left(\bigcap_{\sigma\in A_q\setminus\{\sigma_{i}, \sigma_{j}\}} (f(\sigma)-t_0,f(\sigma)+t_0)\right)$$
and $H \cap \mathrm{im}(f_\infty)=\{q\}$. 

Redefining $f_\infty$ on $A_q$ by any injective assignment of values in $H$ respecting dimension (i.e., faces of a simplex are assigned smaller values than the simplex) we obtain an injective filtration function at a distance at most $t$ from $f$. As a result, any dimension-respecting permutation of elements in $A_q$, nested between other simplices as determined by $f_\infty$, determines a permutation in $\Pi_{t}$. We will call such a permutation a  \textbf{local} $A_q$-\textbf{perturbation} of $\nu$.
On the other hand, redefining $f_\infty$ on $A_q$ by any injective assignment of values in $H$ respecting dimension such that $f_\infty(\sigma_{i}) < q < f_\infty(\sigma_{j})$, we obtain an injective filtration function at distance at most $t_0$ from $f$. As a result, any dimension respecting permutation of elements in $A_q$ nested between other simplices as determined by $f_\infty$, in which $\sigma_{i}$ appears before $\sigma_{j}$, determines a permutation in $\Pi_{t_0}$. We will call such a permutation a  \textbf{restricted local} $A_q$-\textbf{perturbation} of $\nu$. Roughly speaking, the difference between $\Pi_{t_0}$ and $\Pi_{t}$ is that the later may swap $\sigma_{i}$ and $\sigma_{j}$. Observe also that $f(A_q) \subset [f(\sigma_i), f(\sigma_j)]$.
 
Recall that $\Sigma_{t_0}=\{\Delta_1\}$ and fix a permutation $\nu\in\Pi_{t} \setminus \Pi_{t_0}$ (and the corresponding adjusted limit $f_\infty$ of filtration functions) such that the corresponding simplex terminating $\alpha$ is $\Delta_2\neq \Delta_1$. Observe that $\nu$ swaps $\sigma_{i}, \sigma_{j}$, i.e., $\sigma_{i}$ appears after $\sigma_{j}$ in this permutation.
 
\begin{enumerate}
\item[(1)] We will now show that $\Delta_1\in \{\sigma_{i}, \sigma_{j}\}.$ Assume that $\Delta_1\notin \{\sigma_{i}, \sigma_{j}\}.$ 
\begin{enumerate}
 \item Then $\Delta_1$ may appear as the first of the simplices of $A_q$ in some restricted local $A_q$ perturbation of $\nu$. In particular, $\Delta_1$ and the simplices appearing before $A_q$ contain a nullhomology of $[\alpha]$.
 \item  On the other hand, $\Delta_1$ may appear as the last of the simplices of $A_q$ in some restricted local $A_q$ perturbation of $\nu$. In particular, $A_q\setminus\{\Delta_1\}$ and the simplices appearing before $A_q$ do not contain a nullhomology of $[\alpha]$.
\end{enumerate}

As a result, $\Sigma_{t}$ can only contain $\Delta_1$ as in permutation $\nu$, class $[\alpha]$ becomes trivial by the time $\Delta_1$ is added (by (a)) but not before (by (b)). This is a contradiction. This proves our claim.
\item[(2)]  If $\Delta_1$ was not maximal but rather a codimension $1$ face of a simplex $\widehat \Delta$ in $K^f_{f(\Delta_1)+2s}$, then by Corollary \ref{switchmany} we  could construct an injective filtration function $g$ of $K$ with $||f-g||_\infty < t_0$ in which the boundary of $\widehat \Delta$ would appear before $\Delta_1$, while the simplices in $K^f$ appearing before $\Delta_1$ would also appear before $\Delta_1$ in $K^g$. This would mean that $\Delta_1$ would not be the simplex terminating $[\alpha]$ in $K^g$ as it could be replaced by the same-dimensional simplices of $\di \widehat \Delta$. This contradicts the fact that $\Sigma_{t_0}=\{\Delta_1\}$.
\item[(3)] 
Let $\Delta_1=\sigma_{j}$.
\begin{enumerate}
	\item There is a restricted local $A_q$ perturbation of $\nu$ in which $\sigma_{j}$ appears as the last simplex. As $\Sigma_{t_0}=\{\sigma_{j}\}$, the simplices of $A_q\setminus\{\sigma_{j}\}$ and all the simplices appearing before them do not contain a nullhomology of $\alpha$, i.e., each nullhomology induced by a local $A_q$-perturbation of $\nu$ contains $\sigma_{j}$.
	\item On the other hand, there is a restricted local $A_q$ perturbation of $\nu$ in which $\sigma_{j}$ appears as the second simplex, right after $\sigma_{i}$. This means that $\sigma_{i}, \sigma_{j}$ and the simplices appearing before $A_q$ contain a nullhomology of $\alpha$. 
	\item There is a local $A_q$-perturbation of $\nu$ in which $\sigma_{i}$ appears as the second simplex, right after $\sigma_{j}$. By (2) the terminal simplex of $\alpha$ in this perturbation is either $\sigma_i$ or $\sigma_j$. If $\sigma_{j}$ was the terminal simplex of $\alpha$ in this perturbation,  $ \sigma_{j}$ and the simplices appearing before $A_q$ would contain a nullhomology of $\alpha$. 
Consequently (a) would imply the terminal simplex of $\alpha$ in any local $A_q$ perturbation of $\nu$ is $\sigma_j$ implying $|\Sigma_{t}|=1,$ a contradiction. According to (b), $\sigma_{i}$ appears as the terminal simplex for this permutation implying $\{\sigma_{i}, \sigma_{j}\} \subseteq \Sigma_{t}$.
  \end{enumerate}
  
If $\sigma_{i}$ was a birth simplex, the term $\partial \sigma_{i}$ could have been replaced by a combination of boundaries of simplices appearing before $A_q$. We would thus transform the nullhomology mentioned in (c), which consists of terms containing $\sigma_{i}, \sigma_{j}$ and simplices appearing before $A_q$, into a nullhomology consisting of terms  $\sigma_{j}$ and simplices appearing before $A_q$. Such a nullhomology does not exist, as was mentioned in (c), so $\sigma_{i}$ must be a terminal simplex in $K^f$.
\item[(4)] Let $\Delta_1=\sigma_{i}$.

\begin{enumerate}
 	\item There is a restricted local $A_q$-perturbation of $\nu$ in which $\sigma_{j}$ appears as the last simplex and $\sigma_{i}$ appears as the second simplex from the last. As $\Sigma_{t_0}=\{\sigma_{i}\}$, the simplices of $A_q\setminus\{\sigma_{i}, \sigma_{j}\}$ and all the simplices appearing before them do not contain a nullhomology of $\alpha$, i.e., each nullhomology induced by a local $A_q$-perturbation of $\nu$ contains $\sigma_{j}$ or $\sigma_{i}$. 
	\item On the other hand, there is a restricted local $A_q$-perturbation of $\nu$ in which $\sigma_{i}$ appears as the first simplex. This means that $\sigma_{i}$ and the simplices appearing before $A_q$ contain a nullhomology of $\alpha$.
	\item There is a local $A_q$-perturbation of $\nu$ in which $\sigma_{i}$ appears as the last simplex and $\sigma_{j}$ appears as the second simplex from the last. By (a) the terminal simplex of $\alpha$ in this perturbation is either $\sigma_i$ or $\sigma_j$. If $\sigma_{i}$ was the terminal simplex of $\alpha$ in this perturbation, then simplices of $A_q\setminus\{\sigma_{i}\}$ and all the simplices appearing before them do not contain a nullhomology of $\alpha$. By (b) this would mean $|\Sigma_{t}|=1,$ a contradiction. Thus $\sigma_{j}$ appears as the terminal simplex for this permutation implying $\{\sigma_{i}, \sigma_{j}\} \subseteq \Sigma_{t}$.
\end{enumerate}

It remains to prove that $\sigma_{j}$ is a birth simplex in $K^f$. 

By (b) we have 
$$
\alpha = \di \mu_1 \sigma_i +\di \sum_{k} \lambda_k \tau_k,
$$
with $\mu_1 \neq 0$ and $\tau_k$ being simplices appearing before $A_q$. On the other hand, (c) implies 
$$
\alpha = \di \mu_2 \sigma_j +\di \sum_{l} \lambda'_l \tau'_l,
$$
with $\mu_2 \neq 0$ and each $\tau'_l$ being a simplex from $A_q$ or appearing before $A_q$. Subtracting the equations we obtain
$$
\di \mu_2 \sigma_j =\di\Big(- \sum_{l} \lambda'_l \tau'_l+ \mu_1 \sigma_i + \sum_{k} \lambda_k \tau_k\Big).
$$
Recall that $f(A_q) \subset [f(\sigma_i), f(\sigma_j)]$ and thus the $f$-values of simplices $\tau_k$, $\tau'_l$ and $\sigma_i$ are below $f(\sigma_j)$. 
The last equality thus implies $\sigma_{j}$ is a birth simplex in $K^f$.
\end{enumerate}
\end{proof}

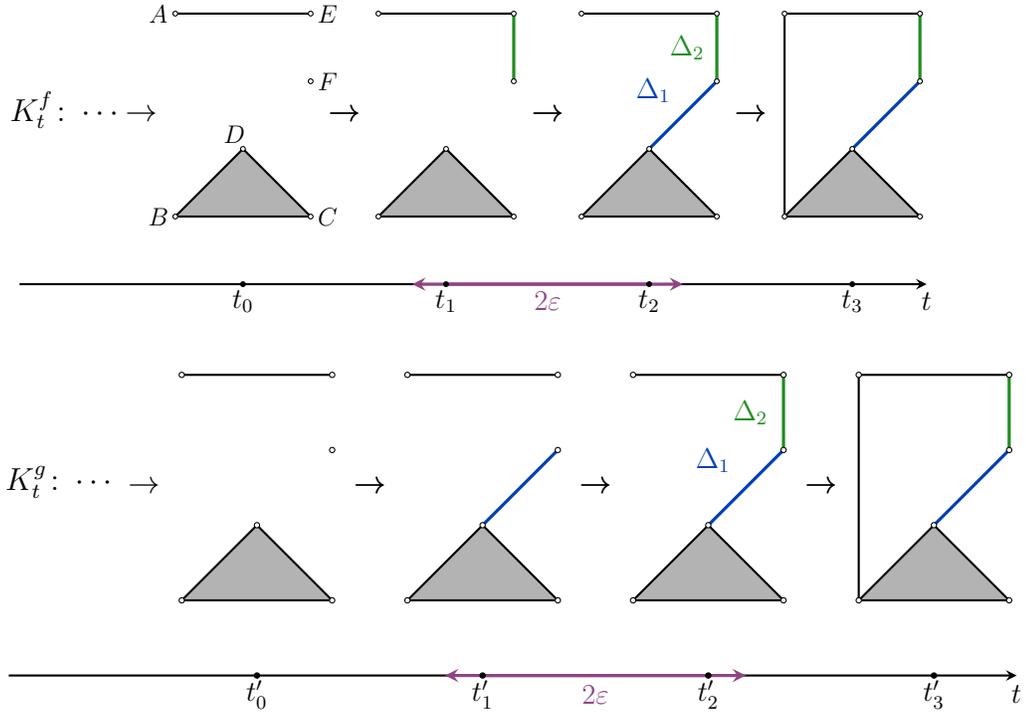
\begin{figure}[!ht]
\begin{center}\begin{tikzpicture}[outer sep=0.8pt, scale=.9, baseline=1.0cm]
\def\R{4}\def\r{0.75}\def\m{1.5}\def\kot{-60}\def\l{0.25}
\path (-1.6,3) coordinate (A0) (-1.6,0) coordinate (B0)
      (-1,3) coordinate (C0) (-1,0) coordinate (D0)
      (0,3) coordinate (A) (0,0) coordinate (B)
	  (2,0) coordinate (C) (1,1) coordinate (D)
	  (2,3) coordinate (E) (2,2) coordinate (F);	  
\draw[thick, fill=black!30!white] (B) -- (C) -- (D) -- cycle;
\draw[thick] (A) -- (E);
\draw[fill=black!0] (A) circle (1pt);\draw[fill=black!0] (B) circle (1pt);
\draw[fill=black!0] (C) circle (1pt);\draw[fill=black!0] (D) circle (1pt);
\draw[fill=black!0] (E) circle (1pt);\draw[fill=black!0] (F) circle (1pt);
\node at ($(A)+(180:\l)$) {\scalebox{0.8}{$A$}};
\node at ($(D)+(120:\l)$) {\scalebox{0.8}{$D$}};
\node at ($(E)+(0:\l)$) {\scalebox{0.8}{$E$}};
\node at ($(F)+(0:\l)$) {\scalebox{0.8}{$F$}};
\node at ($(B)+(180:\l)$) {\scalebox{0.8}{$B$}};
\node at ($(C)+(0:\l)$) {\scalebox{0.8}{$C$}};
\path (3,3) coordinate (A1) (3,0) coordinate (B1)
	  (5,0) coordinate (C1) (4,1) coordinate (D1)
	  (5,3) coordinate (E1) (5,2) coordinate (F1);	  
\draw[thick, fill=black!30!white] (B1) -- (C1) -- (D1) -- cycle;
\draw[thick] (A1) -- (E1);
\draw[very thick, color=zelena] (E1) -- (F1);
\draw[fill=black!0] (A1) circle (1pt);\draw[fill=black!0] (B1) circle (1pt);
\draw[fill=black!0] (C1) circle (1pt);\draw[fill=black!0] (D1) circle (1pt);
\draw[fill=black!0] (E1) circle (1pt);\draw[fill=black!0] (F1) circle (1pt);
\path (6,3) coordinate (A2) (6,0) coordinate (B2)
	  (8,0) coordinate (C2) (7,1) coordinate (D2)
	  (8,3) coordinate (E2) (8,2) coordinate (F2);	  
\draw[thick, fill=black!30!white] (B2) -- (C2) -- (D2) -- cycle;
\draw[thick] (A2) -- (E2);
\draw[very thick, color=zelena] (E2) -- (F2);
\draw[very thick, color=modra] (F2) -- (D2);
\draw[fill=black!0] (A2) circle (1pt);\draw[fill=black!0] (B2) circle (1pt);
\draw[fill=black!0] (C2) circle (1pt);\draw[fill=black!0] (D2) circle (1pt);
\draw[fill=black!0] (E2) circle (1pt);\draw[fill=black!0] (F2) circle (1pt);
\node[anchor=south east] at ($(D2)!0.5!(F2)$) {\color{modra}\scalebox{0.9}{$\Delta_1$}};
\node[anchor=east] at ($(E2)!0.5!(F2)$) {\color{zelena}\scalebox{0.9}{$\Delta_2$}};
\path (9,3) coordinate (A3) (9,0) coordinate (B3)
	  (11,0) coordinate (C3) (10,1) coordinate (D3)
	  (11,3) coordinate (E3) (11,2) coordinate (F3);	  
\draw[thick, fill=black!30!white] (B3) -- (C3) -- (D3) -- cycle;
\draw[thick] (A3) -- (E3);
\draw[very thick, color=zelena] (E3) -- (F3);
\draw[very thick, color=modra] (F3) -- (D3);
\draw[thick] (A3) -- (B3);
\draw[fill=black!0] (A3) circle (1pt);\draw[fill=black!0] (B3) circle (1pt);
\draw[fill=black!0] (C3) circle (1pt);\draw[fill=black!0] (D3) circle (1pt);
\draw[fill=black!0] (E3) circle (1pt);\draw[fill=black!0] (F3) circle (1pt);
\path ($(C)!0.5!(E)$) coordinate (R1)
	  ($(A1)!0.5!(B1)$) coordinate (R2)
	  ($(C1)!0.5!(E1)$) coordinate (R3)
	  ($(A2)!0.5!(B2)$) coordinate (R4)
	  ($(C2)!0.5!(E2)$) coordinate (R5)
	  ($(A3)!0.5!(B3)$) coordinate (R6);
\node at ($(R1)!0.5!(R2)$) {\scalebox{1}{$\rightarrow$}};
\node at ($(R3)!0.5!(R4)$) {\scalebox{1}{$\rightarrow$}};
\node at ($(R5)!0.5!(R6)$) {\scalebox{1}{$\rightarrow$}};
\path ($(C0)!0.5!(D0)$) coordinate (S0) ($(A)!0.5!(B)$) coordinate (S1)
      ($(C)!0.5!(E)$) coordinate (R1) ($(A1)!0.5!(B1)$) coordinate (R2)
	  ($(C1)!0.5!(E1)$) coordinate (R3) ($(A2)!0.5!(B2)$) coordinate (R4)
	  ($(C2)!0.5!(E2)$) coordinate (R5) ($(A3)!0.5!(B3)$) coordinate (R6);
\node at ($(S0)!0.5!(S1)$) {\scalebox{1}{$\rightarrow$}};
\node at ($(R1)!0.5!(R2)$) {\scalebox{1}{$\rightarrow$}};
\node at ($(R3)!0.5!(R4)$) {\scalebox{1}{$\rightarrow$}};
\node at ($(R5)!0.5!(R6)$) {\scalebox{1}{$\rightarrow$}};
\node at ($(A0)!0.5!(B0)+(90:0.3*\l)$) {\scalebox{1}{$K_t^f\colon \cdots$}};
\path (-2.3,-1) coordinate (A4) (11.1,-1) coordinate (B4)
	  (1,-1) coordinate (C4) (4,-1) coordinate (D4)
	  (7,-1) coordinate (E4) (10,-1) coordinate (F4);
\draw[thick, stealth-] (B4) -- (A4);
\draw[very thick, stealth-, color=sliva] ($(D4)+(180:2*\l)$) -- (E4);
\draw[very thick, stealth-, color=sliva] ($(E4)+(0:2*\l)$) -- (D4);
\node at ($(B4)+(270:\l)$) {\scalebox{0.9}{$t$}};
\node at ($(D4)!0.5!(E4)+(270:\l)$) {\color{sliva}\scalebox{0.9}{$2\e$}};
\draw[fill=black] (C4) circle (1pt);\draw[fill=black] (D4) circle (1pt);
\draw[fill=black] (E4) circle (1pt);\draw[fill=black] (F4) circle (1pt);
\node at ($(C4)+(270:\l)$) {\scalebox{0.9}{$t_0$}};
\node at ($(D4)+(270:\l)$) {\scalebox{0.9}{$t_1$}};
\node at ($(E4)+(270:\l)$) {\scalebox{0.9}{$t_2$}};
\node at ($(F4)+(270:\l)$) {\scalebox{0.9}{$t_3$}};
\end{tikzpicture}\end{center}
\vskip.5cm
\begin{center}\begin{tikzpicture}[outer sep=2pt, scale=1.0, baseline=1.4cm]
\def\R{4}\def\r{0.75}\def\m{1.5}\def\kot{-60}\def\l{0.25}
\path (-1.6,3) coordinate (A0) (-1.6,0) coordinate (B0)
      (-1,3) coordinate (C0) (-1,0) coordinate (D0)
      (0,3) coordinate (A) (0,0) coordinate (B)
	  (2,0) coordinate (C) (1,1) coordinate (D)
	  (2,3) coordinate (E) (2,2) coordinate (F);	  
\draw[thick, fill=black!30!white] (B) -- (C) -- (D) -- cycle;
\draw[thick] (A) -- (E);
\draw[fill=black!0] (A) circle (1pt);\draw[fill=black!0] (B) circle (1pt);
\draw[fill=black!0] (C) circle (1pt);\draw[fill=black!0] (D) circle (1pt);
\draw[fill=black!0] (E) circle (1pt);\draw[fill=black!0] (F) circle (1pt);
\path (3,3) coordinate (A1) (3,0) coordinate (B1)
	  (5,0) coordinate (C1) (4,1) coordinate (D1)
	  (5,3) coordinate (E1) (5,2) coordinate (F1);	  
\draw[thick, fill=black!30!white] (B1) -- (C1) -- (D1) -- cycle;
\draw[thick] (A1) -- (E1);
\draw[very thick, color=modra] (F1) -- (D1);
\draw[fill=black!0] (A1) circle (1pt);\draw[fill=black!0] (B1) circle (1pt);
\draw[fill=black!0] (C1) circle (1pt);\draw[fill=black!0] (D1) circle (1pt);
\draw[fill=black!0] (E1) circle (1pt);\draw[fill=black!0] (F1) circle (1pt);
\path (6,3) coordinate (A2) (6,0) coordinate (B2)
	  (8,0) coordinate (C2) (7,1) coordinate (D2)
	  (8,3) coordinate (E2) (8,2) coordinate (F2);	  
\draw[thick, fill=black!30!white] (B2) -- (C2) -- (D2) -- cycle;
\draw[thick] (A2) -- (E2);
\draw[very thick, color=zelena] (E2) -- (F2);
\draw[very thick, color=modra] (F2) -- (D2);
\draw[fill=black!0] (A2) circle (1pt);\draw[fill=black!0] (B2) circle (1pt);
\draw[fill=black!0] (C2) circle (1pt);\draw[fill=black!0] (D2) circle (1pt);
\draw[fill=black!0] (E2) circle (1pt);\draw[fill=black!0] (F2) circle (1pt);
\node[anchor=south east] at ($(D2)!0.5!(F2)$) {\color{modra}\scalebox{0.9}{$\Delta_1$}};
\node[anchor=east] at ($(E2)!0.5!(F2)$) {\color{zelena}\scalebox{0.9}{$\Delta_2$}};
\path (9,3) coordinate (A3) (9,0) coordinate (B3)
	  (11,0) coordinate (C3) (10,1) coordinate (D3)
	  (11,3) coordinate (E3) (11,2) coordinate (F3);	  
\draw[thick, fill=black!30!white] (B3) -- (C3) -- (D3) -- cycle;
\draw[thick] (A3) -- (E3);
\draw[very thick, color=zelena] (E3) -- (F3);
\draw[very thick, color=modra] (F3) -- (D3);
\draw[thick] (A3) -- (B3);
\draw[fill=black!0] (A3) circle (1pt);\draw[fill=black!0] (B3) circle (1pt);
\draw[fill=black!0] (C3) circle (1pt);\draw[fill=black!0] (D3) circle (1pt);
\draw[fill=black!0] (E3) circle (1pt);\draw[fill=black!0] (F3) circle (1pt);
\path ($(C)!0.5!(E)$) coordinate (R1)
	  ($(A1)!0.5!(B1)$) coordinate (R2)
	  ($(C1)!0.5!(E1)$) coordinate (R3)
	  ($(A2)!0.5!(B2)$) coordinate (R4)
	  ($(C2)!0.5!(E2)$) coordinate (R5)
	  ($(A3)!0.5!(B3)$) coordinate (R6);
\node at ($(R1)!0.5!(R2)$) {\scalebox{1}{$\rightarrow$}};
\node at ($(R3)!0.5!(R4)$) {\scalebox{1}{$\rightarrow$}};
\node at ($(R5)!0.5!(R6)$) {\scalebox{1}{$\rightarrow$}};
\path ($(C0)!0.5!(D0)$) coordinate (S0) ($(A)!0.5!(B)$) coordinate (S1)
      ($(C)!0.5!(E)$) coordinate (R1) ($(A1)!0.5!(B1)$) coordinate (R2)
	  ($(C1)!0.5!(E1)$) coordinate (R3) ($(A2)!0.5!(B2)$) coordinate (R4)
	  ($(C2)!0.5!(E2)$) coordinate (R5) ($(A3)!0.5!(B3)$) coordinate (R6);
\node at ($(S0)!0.5!(S1)$) {\scalebox{1}{$\rightarrow$}};
\node at ($(R1)!0.5!(R2)$) {\scalebox{1}{$\rightarrow$}};
\node at ($(R3)!0.5!(R4)$) {\scalebox{1}{$\rightarrow$}};
\node at ($(R5)!0.5!(R6)$) {\scalebox{1}{$\rightarrow$}};
\node at ($(A0)!0.5!(B0)+(90:0.3*\l)$) {\scalebox{1}{$K_t^g\colon \cdots$}};
\path (-2.3,-1) coordinate (A4) (11.1,-1) coordinate (B4)
	  (1,-1) coordinate (C4) (4,-1) coordinate (D4)
	  (7,-1) coordinate (E4) (10,-1) coordinate (F4);
\draw[thick, stealth-] (B4) -- (A4);
\draw[very thick, stealth-, color=sliva] ($(D4)+(180:2*\l)$) -- (E4);
\draw[very thick, stealth-, color=sliva] ($(E4)+(0:2*\l)$) -- (D4);
\node at ($(B4)+(270:\l)$) {\scalebox{0.9}{$t$}};
\node at ($(D4)!0.5!(E4)+(270:\l)$) {\color{sliva}\scalebox{0.9}{$2\e$}};
\draw[fill=black] (C4) circle (1pt);\draw[fill=black] (D4) circle (1pt);
\draw[fill=black] (E4) circle (1pt);\draw[fill=black] (F4) circle (1pt);
\node at ($(C4)+(270:\l)$) {\scalebox{0.9}{$t'_0$}};
\node at ($(D4)+(270:\l)$) {\scalebox{0.9}{$t'_1$}};
\node at ($(E4)+(270:\l)$) {\scalebox{0.9}{$t'_2$}};
\node at ($(F4)+(270:\l)$) {\scalebox{0.9}{$t'_3$}};
\end{tikzpicture}\end{center}
\caption{\small Consider the filtrations $K^f$ and $K^g$ shown above. Let $\a = [B-A]\in H_{0}(K)$, $\Delta_1 = DF$ and $\Delta_2 = EF$. Assume that $|f-g|<\e$, $t_2-t_1<2\e$ and $t'_2-t'_1<2\e$. Then $\Sigma_{\e} = \{\Delta_{1},\Delta_{2}\}$. In the first case $\a$ is terminated by $\Delta_1$ at $t = t_2$. In the second case $\a$ is terminated by $\Delta_2$ at $t = t'_2$. We say that $\Delta_1$ and $\Delta_2$ are \textbf{sequentially critical}: both need to be present to terminate $\a$, the one that appears first terminates another class in the same dimension and the one that appears second is the one that terminates $\a$.}
\label{example5figure2}
\end{figure}
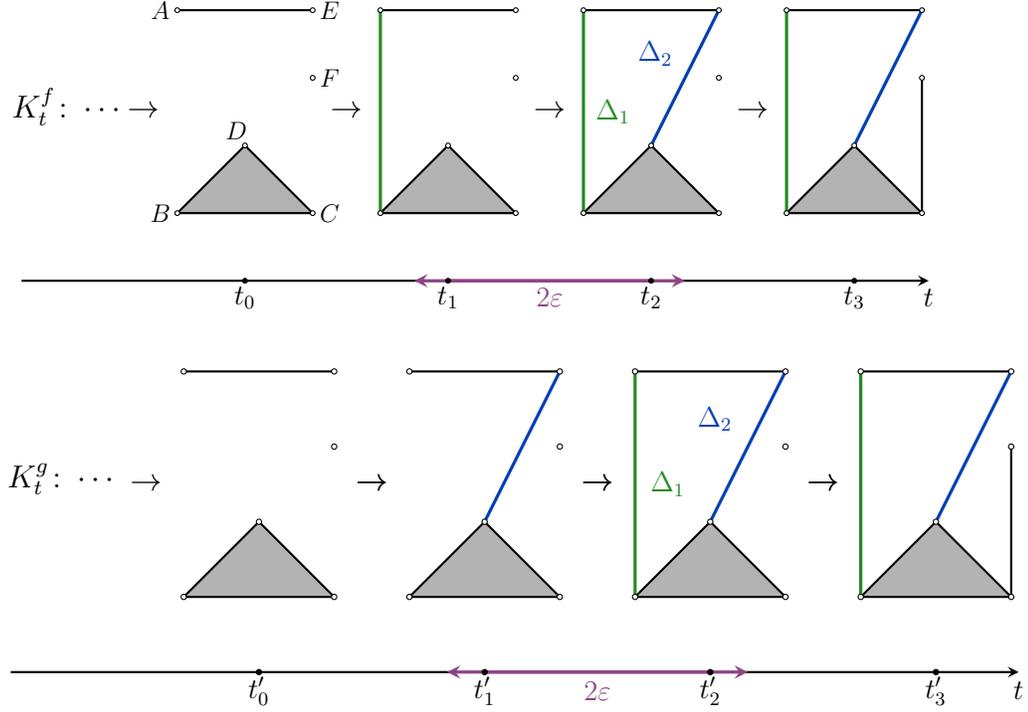
\begin{figure}[!ht]
\begin{center}
\begin{tikzpicture}[outer sep=0.8pt, scale=.9, baseline=1.4cm]
\def\R{4}\def\r{0.75}\def\m{1.5}\def\kot{-60}\def\l{0.25}
\path (-1.6,3) coordinate (A0) (-1.6,0) coordinate (B0)
      (-1,3) coordinate (C0) (-1,0) coordinate (D0)
      (0,3) coordinate (A) (0,0) coordinate (B)
	  (2,0) coordinate (C) (1,1) coordinate (D)
	  (2,3) coordinate (E) (2,2) coordinate (F);	  
\draw[thick, fill=black!30!white] (B) -- (C) -- (D) -- cycle;
\draw[thick] (A) -- (E);
\draw[fill=black!0] (A) circle (1pt);\draw[fill=black!0] (B) circle (1pt);
\draw[fill=black!0] (C) circle (1pt);\draw[fill=black!0] (D) circle (1pt);
\draw[fill=black!0] (E) circle (1pt);\draw[fill=black!0] (F) circle (1pt);
\node at ($(A)+(180:\l)$) {\scalebox{0.8}{$A$}};
\node at ($(D)+(120:\l)$) {\scalebox{0.8}{$D$}};
\node at ($(E)+(0:\l)$) {\scalebox{0.8}{$E$}};
\node at ($(F)+(0:\l)$) {\scalebox{0.8}{$F$}};
\node at ($(B)+(180:\l)$) {\scalebox{0.8}{$B$}};
\node at ($(C)+(0:\l)$) {\scalebox{0.8}{$C$}};
\path (3,3) coordinate (A1) (3,0) coordinate (B1)
	  (5,0) coordinate (C1) (4,1) coordinate (D1)
	  (5,3) coordinate (E1) (5,2) coordinate (F1);	  
\draw[thick, fill=black!30!white] (B1) -- (C1) -- (D1) -- cycle;
\draw[thick] (A1) -- (E1);
\draw[very thick, color=zelena] (A1) -- (B1);
\draw[fill=black!0] (A1) circle (1pt);\draw[fill=black!0] (B1) circle (1pt);
\draw[fill=black!0] (C1) circle (1pt);\draw[fill=black!0] (D1) circle (1pt);
\draw[fill=black!0] (E1) circle (1pt);\draw[fill=black!0] (F1) circle (1pt);
\path (6,3) coordinate (A2) (6,0) coordinate (B2)
	  (8,0) coordinate (C2) (7,1) coordinate (D2)
	  (8,3) coordinate (E2) (8,2) coordinate (F2);	  
\draw[thick, fill=black!30!white] (B2) -- (C2) -- (D2) -- cycle;
\draw[thick] (A2) -- (E2);
\draw[very thick, color=zelena] (A2) -- (B2);
\draw[very thick, color=modra] (E2) -- (D2);
\draw[fill=black!0] (A2) circle (1pt);\draw[fill=black!0] (B2) circle (1pt);
\draw[fill=black!0] (C2) circle (1pt);\draw[fill=black!0] (D2) circle (1pt);
\draw[fill=black!0] (E2) circle (1pt);\draw[fill=black!0] (F2) circle (1pt);
\node[anchor=south east] at ($(D2)!0.5!(E2)$) {\color{modra}\scalebox{0.9}{$\Delta_2$}};
\node[anchor=west] at ($(A2)!0.5!(B2)$) {\color{zelena}\scalebox{0.9}{$\Delta_1$}};
\path (9,3) coordinate (A3) (9,0) coordinate (B3)
	  (11,0) coordinate (C3) (10,1) coordinate (D3)
	  (11,3) coordinate (E3) (11,2) coordinate (F3);	  
\draw[thick, fill=black!30!white] (B3) -- (C3) -- (D3) -- cycle;
\draw[thick] (A3) -- (E3);
\draw[very thick, color=zelena] (A3) -- (B3);
\draw[very thick, color=modra] (E3) -- (D3);
\draw[thick] (C3) -- (F3);
\draw[fill=black!0] (A3) circle (1pt);\draw[fill=black!0] (B3) circle (1pt);
\draw[fill=black!0] (C3) circle (1pt);\draw[fill=black!0] (D3) circle (1pt);
\draw[fill=black!0] (E3) circle (1pt);\draw[fill=black!0] (F3) circle (1pt);
\path ($(C0)!0.5!(D0)$) coordinate (S0) ($(A)!0.5!(B)$) coordinate (S1)
      ($(C)!0.5!(E)$) coordinate (R1) ($(A1)!0.5!(B1)$) coordinate (R2)
	  ($(C1)!0.5!(E1)$) coordinate (R3) ($(A2)!0.5!(B2)$) coordinate (R4)
	  ($(C2)!0.5!(E2)$) coordinate (R5) ($(A3)!0.5!(B3)$) coordinate (R6);
\node at ($(S0)!0.5!(S1)$) {\scalebox{1}{$\rightarrow$}};
\node at ($(R1)!0.5!(R2)$) {\scalebox{1}{$\rightarrow$}};
\node at ($(R3)!0.5!(R4)$) {\scalebox{1}{$\rightarrow$}};
\node at ($(R5)!0.5!(R6)$) {\scalebox{1}{$\rightarrow$}};
\node at ($(A0)!0.5!(B0)+(90:0.3*\l)$) {\scalebox{1}{$K_t^f\colon \cdots$}};
\path (-2.3,-1) coordinate (A4) (11.1,-1) coordinate (B4)
	  (1,-1) coordinate (C4) (4,-1) coordinate (D4)
	  (7,-1) coordinate (E4) (10,-1) coordinate (F4);
\draw[thick, stealth-] (B4) -- (A4);
\draw[very thick, stealth-, color=sliva] ($(D4)+(180:2*\l)$) -- (E4);
\draw[very thick, stealth-, color=sliva] ($(E4)+(0:2*\l)$) -- (D4);
\node at ($(B4)+(270:\l)$) {\scalebox{0.9}{$t$}};
\node at ($(D4)!0.5!(E4)+(270:\l)$) {\color{sliva}\scalebox{0.9}{$2\e$}};
\draw[fill=black] (C4) circle (1pt);\draw[fill=black] (D4) circle (1pt);
\draw[fill=black] (E4) circle (1pt);\draw[fill=black] (F4) circle (1pt);
\node at ($(C4)+(270:\l)$) {\scalebox{0.9}{$t_0$}};
\node at ($(D4)+(270:\l)$) {\scalebox{0.9}{$t_1$}};
\node at ($(E4)+(270:\l)$) {\scalebox{0.9}{$t_2$}};
\node at ($(F4)+(270:\l)$) {\scalebox{0.9}{$t_3$}};
\end{tikzpicture}\end{center}
\vskip.5cm
\begin{center}\begin{tikzpicture}[outer sep=2pt, scale=1.0, baseline=1.4cm]
\def\R{4}\def\r{0.75}\def\m{1.5}\def\kot{-60}\def\l{0.25}
\path (-1.6,3) coordinate (A0) (-1.6,0) coordinate (B0)
      (-1,3) coordinate (C0) (-1,0) coordinate (D0)
      (0,3) coordinate (A) (0,0) coordinate (B)
	  (2,0) coordinate (C) (1,1) coordinate (D)
	  (2,3) coordinate (E) (2,2) coordinate (F);	  
\draw[thick, fill=black!30!white] (B) -- (C) -- (D) -- cycle;
\draw[thick] (A) -- (E);
\draw[fill=black!0] (A) circle (1pt);\draw[fill=black!0] (B) circle (1pt);
\draw[fill=black!0] (C) circle (1pt);\draw[fill=black!0] (D) circle (1pt);
\draw[fill=black!0] (E) circle (1pt);\draw[fill=black!0] (F) circle (1pt);
\path (3,3) coordinate (A1) (3,0) coordinate (B1)
	  (5,0) coordinate (C1) (4,1) coordinate (D1)
	  (5,3) coordinate (E1) (5,2) coordinate (F1);	  
\draw[thick, fill=black!30!white] (B1) -- (C1) -- (D1) -- cycle;
\draw[thick] (A1) -- (E1);
\draw[very thick, color=modra] (E1) -- (D1);
\draw[fill=black!0] (A1) circle (1pt);\draw[fill=black!0] (B1) circle (1pt);
\draw[fill=black!0] (C1) circle (1pt);\draw[fill=black!0] (D1) circle (1pt);
\draw[fill=black!0] (E1) circle (1pt);\draw[fill=black!0] (F1) circle (1pt);
\path (6,3) coordinate (A2) (6,0) coordinate (B2)
	  (8,0) coordinate (C2) (7,1) coordinate (D2)
	  (8,3) coordinate (E2) (8,2) coordinate (F2);	  
\draw[thick, fill=black!30!white] (B2) -- (C2) -- (D2) -- cycle;
\draw[thick] (A2) -- (E2);
\draw[very thick, color=zelena] (A2) -- (B2);
\draw[very thick, color=modra] (E2) -- (D2);
\draw[fill=black!0] (A2) circle (1pt);\draw[fill=black!0] (B2) circle (1pt);
\draw[fill=black!0] (C2) circle (1pt);\draw[fill=black!0] (D2) circle (1pt);
\draw[fill=black!0] (E2) circle (1pt);\draw[fill=black!0] (F2) circle (1pt);\node[anchor=south east] at ($(D2)!0.5!(E2)$) {\color{modra}\scalebox{0.9}{$\Delta_2$}};
\node[anchor=west] at ($(A2)!0.5!(B2)$) {\color{zelena}\scalebox{0.9}{$\Delta_1$}};
\path (9,3) coordinate (A3) (9,0) coordinate (B3)
	  (11,0) coordinate (C3) (10,1) coordinate (D3)
	  (11,3) coordinate (E3) (11,2) coordinate (F3);	  
\draw[thick, fill=black!30!white] (B3) -- (C3) -- (D3) -- cycle;
\draw[thick] (A3) -- (E3);
\draw[very thick, color=zelena] (A3) -- (B3);
\draw[very thick, color=modra] (E3) -- (D3);
\draw[thick] (C3) -- (F3);
\draw[fill=black!0] (A3) circle (1pt);\draw[fill=black!0] (B3) circle (1pt);
\draw[fill=black!0] (C3) circle (1pt);\draw[fill=black!0] (D3) circle (1pt);
\draw[fill=black!0] (E3) circle (1pt);\draw[fill=black!0] (F3) circle (1pt);
\path ($(C)!0.5!(E)$) coordinate (R1)
	  ($(A1)!0.5!(B1)$) coordinate (R2)
	  ($(C1)!0.5!(E1)$) coordinate (R3)
	  ($(A2)!0.5!(B2)$) coordinate (R4)
	  ($(C2)!0.5!(E2)$) coordinate (R5)
	  ($(A3)!0.5!(B3)$) coordinate (R6);
\node at ($(R1)!0.5!(R2)$) {\scalebox{1}{$\rightarrow$}};
\node at ($(R3)!0.5!(R4)$) {\scalebox{1}{$\rightarrow$}};
\node at ($(R5)!0.5!(R6)$) {\scalebox{1}{$\rightarrow$}};
\path ($(C0)!0.5!(D0)$) coordinate (S0) ($(A)!0.5!(B)$) coordinate (S1)
      ($(C)!0.5!(E)$) coordinate (R1) ($(A1)!0.5!(B1)$) coordinate (R2)
	  ($(C1)!0.5!(E1)$) coordinate (R3) ($(A2)!0.5!(B2)$) coordinate (R4)
	  ($(C2)!0.5!(E2)$) coordinate (R5) ($(A3)!0.5!(B3)$) coordinate (R6);
\node at ($(S0)!0.5!(S1)$) {\scalebox{1}{$\rightarrow$}};
\node at ($(R1)!0.5!(R2)$) {\scalebox{1}{$\rightarrow$}};
\node at ($(R3)!0.5!(R4)$) {\scalebox{1}{$\rightarrow$}};
\node at ($(R5)!0.5!(R6)$) {\scalebox{1}{$\rightarrow$}};
\node at ($(A0)!0.5!(B0)+(90:0.3*\l)$) {\scalebox{1}{$K_t^g\colon \cdots$}};
\path (-2.3,-1) coordinate (A4) (11.1,-1) coordinate (B4)
	  (1,-1) coordinate (C4) (4,-1) coordinate (D4)
	  (7,-1) coordinate (E4) (10,-1) coordinate (F4);
\draw[thick, stealth-] (B4) -- (A4);
\draw[very thick, stealth-, color=sliva] ($(D4)+(180:2*\l)$) -- (E4);
\draw[very thick, stealth-, color=sliva] ($(E4)+(0:2*\l)$) -- (D4);
\node at ($(B4)+(270:\l)$) {\scalebox{0.9}{$t$}};
\node at ($(D4)!0.5!(E4)+(270:\l)$) {\color{sliva}\scalebox{0.9}{$2\e$}};
\draw[fill=black] (C4) circle (1pt);\draw[fill=black] (D4) circle (1pt);
\draw[fill=black] (E4) circle (1pt);\draw[fill=black] (F4) circle (1pt);
\node at ($(C4)+(270:\l)$) {\scalebox{0.9}{$t'_0$}};
\node at ($(D4)+(270:\l)$) {\scalebox{0.9}{$t'_1$}};
\node at ($(E4)+(270:\l)$) {\scalebox{0.9}{$t'_2$}};
\node at ($(F4)+(270:\l)$) {\scalebox{0.9}{$t'_3$}};
\end{tikzpicture}\end{center}
\caption{\small Consider the filtrations $K^f$ and $K^g$ shown above. Let $\a = [B-A]\in H_{0}(K)$, $\Delta_1 = AB$ and $\Delta_2 = DE$. Assume that $|f-g|<\e$, $t_2-t_1<2\e$ and $t'_2-t'_1<2\e$. Then $\Sigma_{\e} = \{\Delta_{1},\Delta_{2}\}$. In the first case $\a$ is terminated by $\Delta_1$ at $t = t_1$. In the second case $\a$ is terminated by $\Delta_2$ at $t = t'_1$. We say that $\Delta_1$ and $\Delta_2$ are \textbf{independently critical}. One of them terminates $\a$ and the other creates a new class in dimension above.}
\label{example5figure1}
\end{figure}
\begin{example}\label{example6}
Consider the following example with $\alpha = [B-A]$, $\Delta_1 = AB$, $\Delta_2 = AC$ and $\Delta_3 = BC$.
If the order of the simplices is
$$A,B,C,AB,BC,AC,$$
then $\alpha$ is terminated by $\Delta_1$. If we assume $f$ assigns to the simplices the values $1$, $2$, $3$, $4$, $5$ and $6$ in this order, then for $\e\leq1$ we have $\Sigma_\e = \{\Delta_1\}$. If the another function changes the order to any where $\Delta_1$ occurs after $\Delta_2$ and $\Delta_3$, namely
$$A,B,C,BC,AC,AB\quad\textrm{ or }\quad A,B,C,AC,BC,AB,$$
then $\alpha$ is terminated by $\Delta_2$ and $\Delta_3$, respectively, and in both cases $\Sigma_\e = \{\Delta_1, \Delta_2, \Delta_3\}$ when $\e>1$. In this case it is not possible for $\Sigma_\e$ to have exactly two elements and $|\Sigma_\e|$ jumps from $1$ to $3$. An analoguous example with $n+2$ vertices and $n+2$ edges (see Figure \ref{example6figure1}) shows that $|\Sigma_\e|$ can jump from $1$ to $n+2$ for any positive integer $n$.
\end{example}

\begin{figure}[ht]
\begin{center}\begin{tikzpicture}[outer sep=2pt, scale=1.2]
\def\R{4}\def\r{0.75}\def\m{1.5}\def\kot{-60}\def\l{0.3}
\path (0,0) coordinate (A)
	  (0,2) coordinate (B)
	  (-1.71,1) coordinate (C);
\draw[thick] (A) -- (B) -- (C) -- cycle;
\node at ($(A)+(-90:\l)$) {\scalebox{0.8}{$A$}};
\node at ($(B)+( 90:\l)$) {\scalebox{0.8}{$B$}};
\node at ($(C)+(180:\l)$) {\scalebox{0.8}{$C$}};
\draw[fill=black!0] (A) circle (1pt);
\draw[fill=black!0] (B) circle (1pt);
\draw[fill=black!0] (C) circle (1pt);
\end{tikzpicture}\qquad\qquad
\begin{tikzpicture}[outer sep=2pt, scale=1.2]
\def\R{4}\def\r{0.75}\def\m{1.5}\def\kot{-60}\def\l{0.3}
\path (0.4,0.5) coordinate (A)
	  (0.4,1.5) coordinate (B)
	  (-0.4,1.9) coordinate (C1)
	  (-1.2,1.5) coordinate (C2)
	  (-1.2,0.5) coordinate (C3)
	  (-0.4,0.1) coordinate (C4)
	  (-1.2,1.08) coordinate (C);
\draw[thick] (C3) -- (C4) -- (A) -- (B) -- (C1) -- (C2);
\node at ($(A)+(-90:\l)$) {\scalebox{0.8}{$A$}};
\node at ($(B)+( 90:\l)$) {\scalebox{0.8}{$B$}};
\node at ($(C)$) {\scalebox{0.8}{$\vdots$}};
\node at ($(C1)+( 90:\l)$) {\scalebox{0.8}{$C_1$}};
\node at ($(C2)+(135:\l)$) {\scalebox{0.8}{$C_2$}};
\node at ($(C3)+(225:\l)$) {\scalebox{0.8}{$C_{n-1}$}};
\node at ($(C4)+(-90:\l)$) {\scalebox{0.8}{$C_{n}$}};
\draw[fill=black!0] (A) circle (1pt);
\draw[fill=black!0] (B) circle (1pt);
\draw[fill=black!0] (C1) circle (1pt);
\draw[fill=black!0] (C2) circle (1pt);
\draw[fill=black!0] (C3) circle (1pt);
\draw[fill=black!0] (C4) circle (1pt);
\end{tikzpicture}
\end{center}
\caption{\small Example \ref{example6}: $|\Sigma_\e|$ can increase by more than 1.}
\label{example6figure1}
\end{figure}
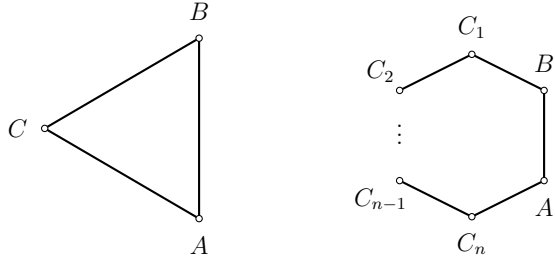

\subsection{Sufficient conditions for rigidity of homology classes}

\begin{theorem}
\label{ThmConseq1}
Given an injective filtration function $f$ on a simplicial complex $K$ and a homology class $[\alpha]\in H_n(K^f_a)$ born at $a$ and terminating at $b$, let:
\begin{itemize}
 \item $R_u =\min\{|f(\tau)-b|;\; \tau \textrm{ a birth $(n+1)$-simplex with } f(\tau)>b\}$ or $\infty$ if undefined;
 \item $R_l =\min\{|f(\tau)-b|;\; \tau \textrm{ a terminal $(n+1)$-simplex with } f(\tau)<b\}$ or $\infty$ if undefined.
\end{itemize}
Then $[\alpha]$ is $\e$-terminally-rigid for $\e=\frac{1}{2}\min\{b-a, R_u, R_l\}.$ 
\end{theorem}

\begin{proof}
For a generic $f$ the conclusion follows from Proposition \ref{PropStep2}. 

Assume $f$ is not generic. For a positive $\delta < \max\{\ir(f)/2,\e/2 \} $ choose a generic function $f_\delta$ such that $||f-f_\delta || \leq \delta$. Quantities $b-a, R_u, R_l$ for this new function are at most $2\delta$ smaller than the original quantities for $f$ and hence $[\alpha]$ in $K^{f_\delta}$  is  $(\e-\delta)$-terminally-rigid.
As any $(\e-2\delta)$-perturbation of $f$ is also an $(\e-\delta)$-perturbation of $f_\delta$, we conclude $f$ is $(\e-2\delta)$-terminally-rigid. As $\delta$ may be arbitrarily small, Lemma \ref{LemmaUSC} concludes the proof.
\end{proof}

\section{Rigidity for barcodes}

In the previous sections we assumed $\e< (b-a)/2$, which ensures that  any $\e$-perturbation $g$ of $f$ still contains a non-trivial homology class $[\alpha]$ in $K^g$. Our main results so far described the change of the terminal simplex of $[\alpha]$ with $\e$-perturbations. The situation is a bit more problematic with barcodes. While the barcodes induced by $f$ and $g$ are at the bottleneck distance at most $\e$, there is no natural way to define the underlying matching \cite{Bauer}. In particular, if $[\alpha]$ represents a bar, the bar matched to it by the isometry theorem may not be represented by $[\alpha]$. On a similar note, given a homology class $[\alpha]$ that is born at $a$ and terminates at $b$, there may be no bar of the form $[a,b)$, see Remark \ref{RemBarcodeIssue} for more details.

\begin{remark}
\label{RemBarcodeIssue}
Let $f$ be an injective filtration function on a simplicial complex $K$ and let $n$ be a positive integer. Assume that for some $a,b\in\mathbb{R}$ an $n$-cycle $\alpha$ is created, so that the corresponding homology class $[\alpha]\in H_n(K_r^f)$ is born at $a$ and terminates at $b$. The creation of $\alpha$ causes the birth of a bar in the persistence diagram. A while later, $\alpha$ might become homologuous to an older cycle $\alpha'$ at a time $c\in (a,b)$, at which point the bar born at $a$ would be terminated according to the elder rule. The class $[\alpha]$ becomes trivial in $H_n(K^f_{b})$ however, so there must exist a bar in the persistence diagram that begins at or before $a$ and dies exactly at $b$. In other words, the starting point of such a bar might be ambiguous, but the moment of its termination is certain. 
\end{remark}

This leads us to focus on a setting in which the same homology class determines the endpoint of a designated bar and of its matched bar arising from $\e$-perturbation. 
 
\begin{theorem}
\label{ThmBarcodeFinal}
Let $f$ be an injective filtration function on a simplicial complex $K$. Assume $[a,b)$ is a bar of $\{H_n(K^f_r)\}_{r\in \R}$ represented by $[\alpha]\in H_n(K^f_a)$  (a homology class born at $a$ and terminating at $b$). Choose $\e < (b-a)/4$ and assume that for all other bars $[a_i,b_i)$ of $\{H_n(K^f_r)\}_{r\in \R}$ either $a_i > a+2\e$ or $b_i < b-2\e$. Let $g$ be an injective filtration function satisfying $||f-g||_\infty \leq \e$, and assume the induced matching matches the bar $[a,b)$ of $\{H_n(K^f_r)\}_{r\in \R}$ to a bar $[a',b')$ of $\{H_n(K^g_r)\}_{r\in \R}$. Then
\begin{enumerate}
 \item[(1)] $[\alpha]$ as a homology class in $\{H_n(K^g_r)\}_{r\in \R}$ terminates at $b'$ and
 \item[(2)] the simplex in $K^f$ terminating the bar $[a,b)$ is the same as the simplex in $K^g$ terminating the bar $[a',b')$ if $\e \leq \frac{1}{2} \min \{R_u,R_l\},$ where: 	
\begin{itemize}
 \item $R_u =\min\{|f(\tau)-b|;\; \tau \textrm{ a birth $(n+1)$-simplex with } f(\tau)>b\}$ or $\infty$ if undefined;
  \item $R_l =\min\{|f(\tau)-b|;\; \tau \textrm{ a terminal $(n+1)$-simplex with } f(\tau)<b\}$ or $\infty$ if undefined.
\end{itemize}
\end{enumerate}
\end{theorem}

\begin{proof}
Let
$$\{H_n(K^g_r)\}_{r\in \R} = \bigoplus_{j\in J} \FF_{[a'_j,b'_j)}$$
be the decomposition into interval modules, indexed so that $[a',b')$ corresponds to $[a'_0,b'_0)$. Our assumption on the barcode of $ \{H_n(K^f_r)\}_{r\in \R}$ and the stability theorem imply the following: 
if $a + \e \in [a'_j,b'_j)$ then either $j=0$ or $b'_j <b - \e$. For each $j$ let $[\alpha_j]$ be the homology class corresponding to the bar $[a'_j,b'_j)$, i.e., $[\alpha_j]$ is born at $a'_j$ and terminates at $b'_j$. 
 
The homology class 
$$[\beta] = \sum_{j\in J} \lambda_j [\alpha_j]\in H_n(K^g_{a + \e})$$
terminates either at:
\begin{itemize}
 \item $b'$, if $\lambda_0 \neq 0$, or
 \item before $b - \e$, if $\lambda_0 = 0$.
\end{itemize}
We now apply this observation to $[\alpha]$. By stability theorem it represents an element of $H_n(K^g_{a + \e})$ and does not terminate before $b - \e$, so it terminates at $b'$ and this concludes the proof of (1).
 
Bars $[a,b)$ of $\{H_n(K^f_r)\}_{r\in \R}$ and $[a',b')$ of $\{H_n(K^g_r)\}_{r\in \R}$ both terminate when $[\alpha]$ terminates. Conclusion (2) now follows from Theorem \ref{ThmConseq1}.
\end{proof}

\begin{remark}
Theorem \ref{ThmBarcodeFinal} provides a sufficient condition on the structure of the barcode that guarantees that the simplex terminating the bar matched to the designated bar $[a,b)$ remains constant through $\e$-perturbations of the filtration function $f$. Going beyond the region of unique terminal simplex within the setting of Theorem \ref{ThmBarcodeFinal}, let $\e_0$ be the maximal $\e$ for which the terminal simplex terminating the bar matched to the designated bar $[a,b)$ remains unique $\Delta_1$ through $\e$-perturbations of the filtration function $f$. Proposition \ref{PropStep2} allows us to deduce at least one additional simplex $\Delta_2$ that appears as the terminal simplex terminating the bar matched to the designated bar $[a,b)$ for some $\e$-perturbations of the filtration function $f$ with $\e > \e_0$. Namely, $\Delta_2$ is an $n$-simplex with the function value $f(\Delta_1) \pm 2\e_0$. If $f$ is generic there is only one such simplex.

To put it differently, we may identify potential simplices generating non-rigidity from the structure of the barcode. For a demonstration see Figure \ref{FigMainIntro1g} and Figure \ref{FigMainIntro2g}. 
\end{remark}

\begin{figure}[htp]
\begin{center}\begin{tikzpicture}[outer sep=1pt, scale=1]
\useasboundingbox (0,-0.1) rectangle (13,5.3);
\def\dx{1.4}\def\dy{0.3}\def\a{1}\def\b{9.5}\def\h{5.5}\def\mid{3}\def\ha{1}
\path (0, -1) coordinate (1);
\draw[color=white, fill=modra, opacity=0.3] (\b-2*\dx,0) -- (\b,0) -- (\b,\mid) -- (\b-2*\dx,\mid) -- cycle;
\draw[color=white, fill=zelena, opacity=0.3] (\b+2*\dx,\h) -- (\b,\h) -- (\b,\mid) -- (\b+2*\dx,\mid) -- cycle;
\draw[-stealth] (0,0) -- (13,0);
\draw (\a,-0.1) -- (\a,0.1);
\draw (\a+\dx,-0.1) -- (\a+\dx,0.1);
\draw (\a+2*\dx,-0.1) -- (\a+2*\dx,0.1);
\draw (\b-2*\dx,-0.1) -- (\b-2*\dx,0.1);
\draw (\b-\dx,-0.1) -- (\b-\dx,0.1);
\draw (\b,-0.1) -- (\b,0.1);
\draw (\b+\dx,-0.1) -- (\b+\dx,0.1);
\draw (\b+2*\dx,-0.1) -- (\b+2*\dx,0.1);
\draw[thick, dotted] (0,\mid) -- (13,\mid);
\draw[thick, dotted] (\a,0) -- (\a,\h);
\draw[thick, dotted] (\a+\dx,0) -- (\a+\dx,\h);
\draw[thick, dotted] (\a+2*\dx,0) -- (\a+2*\dx,\h);
\draw[thick, dotted] (\b-2*\dx,0) -- (\b-2*\dx,\h);
\draw[thick, dotted] (\b-\dx,0) -- (\b-\dx,\h);
\draw[thick, dotted] (\b,0) -- (\b,\h);
\draw[thick, dotted] (\b+\dx,0) -- (\b+\dx,\h);
\draw[thick, dotted] (\b+2*\dx,0) -- (\b+2*\dx,\h);
\node at (0.4,0.5*\mid) {\color{modra}\scalebox{0.8}{$n$}};
\node at (0.4,0.5*\mid+0.5*\h) {\color{zelena}\scalebox{0.8}{$n+1$}};
\node[anchor=east] at (\a,\ha) {\color{sliva}\scalebox{0.8}{$[\alpha]$}};
\node[anchor=north] at (\a,0) {\scalebox{0.8}{$a$}};
\node[anchor=north] at (\a+\dx,0) {\scalebox{0.8}{$a+\e$}};
\node[anchor=north] at (\a+2*\dx,0) {\scalebox{0.8}{$a+2\e$}};
\node[anchor=north] at (\b-2*\dx,0) {\scalebox{0.8}{$b-2\e$}};
\node[anchor=north] at (\b-\dx,0) {\scalebox{0.8}{$b-\e$}};
\node[anchor=north] at (\b,0) {\scalebox{0.8}{$b$}};
\node[anchor=north] at (\b+\dx,0) {\scalebox{0.8}{$b+\e$}};
\node[anchor=north] at (\b+2*\dx,0) {\scalebox{0.8}{$b+2\e$}};
\draw[very thick, color=sliva, -stealth] (\a,\ha) -- (\b,\ha);
\draw[very thick, color=modra, -stealth] (\a-\dx/2,\ha-2*\dy) -- (\b-2.5*\dx,\ha-2*\dy);
\draw[very thick, color=modra, -stealth] (\a-\dx/3,\ha-\dy) -- (\a+1.8*\dx,\ha-\dy);
\draw[very thick, color=modra, -stealth] (\a+\dx/2,\ha+\dy) -- (\b-2.1*\dx,\ha+\dy);
\draw[very thick, color=modra, -stealth] (\a+2.2*\dx,\ha+2*\dy) -- (\b+1.5*\dx,\ha+2*\dy);
\draw[very thick, color=modra, -stealth] (\a+2.5*\dx,\ha+3*\dy) -- (\b+0.5*\dx,\ha+3*\dy);
\draw[very thick, color=modra, -stealth] (\b-1.8*\dx,\ha+4*\dy) -- (\b+1.8*\dx,\ha+4*\dy);
\draw[very thick, color=modra, -stealth] (\b-0.4*\dx,\ha+5*\dy) -- (\b+1.2*\dx,\ha+5*\dy);
\draw[very thick, color=zelena, -stealth] (\a-0.2*\dx,\mid+1*\dy) -- (\b-0.2*\dx,\mid+1*\dy);
\draw[very thick, color=zelena, -stealth] (\a+0.6*\dx,\mid+2*\dy) -- (\b-1.3*\dx,\mid+2*\dy);
\draw[very thick, color=zelena, -stealth] (\a+2.2*\dx,\mid+3*\dy) -- (\b+0.4*\dx,\mid+3*\dy);
\draw[very thick, color=zelena, -stealth] (\b-1.1*\dx,\mid+4*\dy) -- (\b+1.3*\dx,\mid+4*\dy);
\draw[very thick, color=zelena, -stealth] (\b-0.8*\dx,\mid+5*\dy) -- (\b+2.2*\dx,\mid+5*\dy);
\draw[very thick, color=zelena, -stealth] (\b-0.2*\dx,\mid+6*\dy) -- (\b+1.8*\dx,\mid+6*\dy);
\draw[very thick, color=zelena, -stealth] (\b+2.1*\dx,\mid+7*\dy) -- (\b+2.5*\dx,\mid+7*\dy);
\end{tikzpicture}
\end{center}
\caption{\small 
The figure represents persistent homology barcodes of filtration $K^f$ in dimensions $n$ and $n+1$. From the barcodes we can deduce (using Theorem \ref{ThmBarcodeFinal}) that the critical simplex terminating the bar matched with the bar $[a,b)$ is constant (rigid) through all $\e$-perturbation of the filtration function $f$. This conclusion follows from the following facts: (i) no $(n+1)$-dimensional bar is born (in the green area) between $b$ and $b+2\e$, (ii) no $n$-dimensional bar ends (in the blue area) between $b-2\e$ and $b$ (except $\a$ at $b$), and (iii) $\a$ is the only bar that lives through $a+2\e$ and $b-2\e$.
\label{FigMainIntro1g}
}
\end{figure}
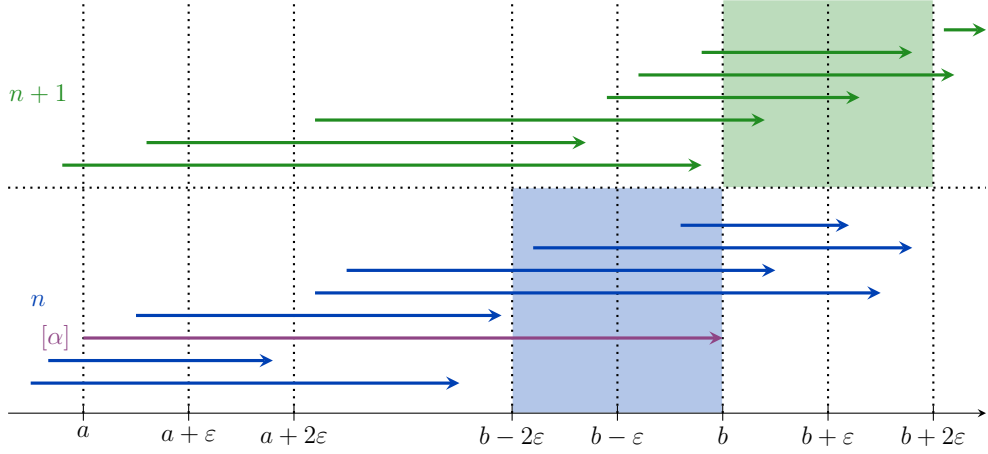
\begin{figure}[htp]
\begin{center}\begin{tikzpicture}[outer sep=1pt, scale=1]
\useasboundingbox (0,-0.1) rectangle (13,5.3);
\def\dx{1.4}\def\dy{0.3}\def\a{1}\def\b{9.5}\def\h{5.5}\def\mid{3}\def\ha{1}
\path (0, -1) coordinate (1);
\draw[color=white, fill=modra, opacity=0.3] (\b-2*\dx,0) -- (\b,0) -- (\b,\mid) -- (\b-2*\dx,\mid) -- cycle;
\draw[color=white, fill=zelena, opacity=0.3] (\b+2*\dx,\h) -- (\b,\h) -- (\b,\mid) -- (\b+2*\dx,\mid) -- cycle;
\draw[-stealth] (0,0) -- (13,0);
\draw (\a,-0.1) -- (\a,0.1);
\draw (\a+\dx,-0.1) -- (\a+\dx,0.1);
\draw (\a+2*\dx,-0.1) -- (\a+2*\dx,0.1);
\draw (\b-2*\dx,-0.1) -- (\b-2*\dx,0.1);
\draw (\b-\dx,-0.1) -- (\b-\dx,0.1);
\draw (\b,-0.1) -- (\b,0.1);
\draw (\b+\dx,-0.1) -- (\b+\dx,0.1);
\draw (\b+2*\dx,-0.1) -- (\b+2*\dx,0.1);
\draw[thick, dotted] (0,\mid) -- (13,\mid);
\draw[thick, dotted] (\a,0) -- (\a,\h);
\draw[thick, dotted] (\a+\dx,0) -- (\a+\dx,\h);
\draw[thick, dotted] (\a+2*\dx,0) -- (\a+2*\dx,\h);
\draw[thick, dotted] (\b-2*\dx,0) -- (\b-2*\dx,\h);
\draw[thick, dotted] (\b-\dx,0) -- (\b-\dx,\h);
\draw[thick, dotted] (\b,0) -- (\b,\h);
\draw[thick, dotted] (\b+\dx,0) -- (\b+\dx,\h);
\draw[thick, dotted] (\b+2*\dx,0) -- (\b+2*\dx,\h);
\node at (0.4,0.5*\mid) {\color{modra}\scalebox{0.8}{$n$}};
\node at (0.4,0.5*\mid+0.5*\h) {\color{zelena}\scalebox{0.8}{$n+1$}};
\node[anchor=east] at (\a,\ha) {\color{sliva}\scalebox{0.8}{$[\alpha]$}};
\node[anchor=north] at (\a,0) {\scalebox{0.8}{$a$}};
\node[anchor=north] at (\a+\dx,0) {\scalebox{0.8}{$a+\e$}};
\node[anchor=north] at (\a+2*\dx,0) {\scalebox{0.8}{$a+2\e$}};
\node[anchor=north] at (\b-2*\dx,0) {\scalebox{0.8}{$b-2\e$}};
\node[anchor=north] at (\b-\dx,0) {\scalebox{0.8}{$b-\e$}};
\node[anchor=north] at (\b,0) {\scalebox{0.8}{$b$}};
\node[anchor=north] at (\b+\dx,0) {\scalebox{0.8}{$b+\e$}};
\node[anchor=north] at (\b+2*\dx,0) {\scalebox{0.8}{$b+2\e$}};
\draw[very thick, color=sliva, -stealth] (\a,\ha) -- (\b,\ha);
\draw[very thick, color=modra, -stealth] (\a-\dx/2,\ha-2*\dy) -- (\b-2.5*\dx,\ha-2*\dy);
\draw[very thick, color=modra, -stealth] (\a-\dx/3,\ha-\dy) -- (\a+1.8*\dx,\ha-\dy);
\draw[very thick, color=modra, -stealth] (\a+\dx/2,\ha+\dy) -- (\b-2*\dx,\ha+\dy);
\draw[very thick, color=modra, -stealth] (\a+2.2*\dx,\ha+2*\dy) -- (\b+1.5*\dx,\ha+2*\dy);
\draw[very thick, color=modra, -stealth] (\a+2.5*\dx,\ha+3*\dy) -- (\b+0.5*\dx,\ha+3*\dy);
\draw[very thick, color=modra, -stealth] (\b-1.8*\dx,\ha+4*\dy) -- (\b+1.8*\dx,\ha+4*\dy);
\draw[very thick, color=modra, -stealth] (\b-0.4*\dx,\ha+5*\dy) -- (\b+1.2*\dx,\ha+5*\dy);
\draw[very thick, color=zelena, -stealth] (\a-0.2*\dx,\mid+1*\dy) -- (\b-0.2*\dx,\mid+1*\dy);
\draw[very thick, color=zelena, -stealth] (\a+0.6*\dx,\mid+2*\dy) -- (\b-1.3*\dx,\mid+2*\dy);
\draw[very thick, color=zelena, -stealth] (\a+2.2*\dx,\mid+3*\dy) -- (\b+0.4*\dx,\mid+3*\dy);
\draw[very thick, color=zelena, -stealth] (\b-1.1*\dx,\mid+4*\dy) -- (\b+1.3*\dx,\mid+4*\dy);
\draw[very thick, color=zelena, -stealth] (\b-0.8*\dx,\mid+5*\dy) -- (\b+2.2*\dx,\mid+5*\dy);
\draw[very thick, color=zelena, -stealth] (\b-0.2*\dx,\mid+6*\dy) -- (\b+1.8*\dx,\mid+6*\dy);
\draw[very thick, color=zelena, -stealth] (\b+2*\dx,\mid+7*\dy) -- (\b+2.5*\dx,\mid+7*\dy);
\draw[fill=black, color=black] (\b+2*\dx,\mid+7*\dy) circle (1.5pt);
\draw[fill=black, color=black] (\b-2*\dx,\ha+\dy) circle (1.5pt);
\end{tikzpicture}
\end{center}
\caption{\small If $\e$ is the scale by which rigidity is broken, we can find a bar in dimension $n+1$ starting at $b+2\e$ or a bar in dimension $n$ ending at $b-2\e$. In this case, for each $\e'>\e$  at least one of the corresponding simplices (appearing at the black dots) appears as the terminal simplex of a bar matched with $[a,b)$ in some $\e'$ perturbation of $f$.}
\label{FigMainIntro2g}
\end{figure}
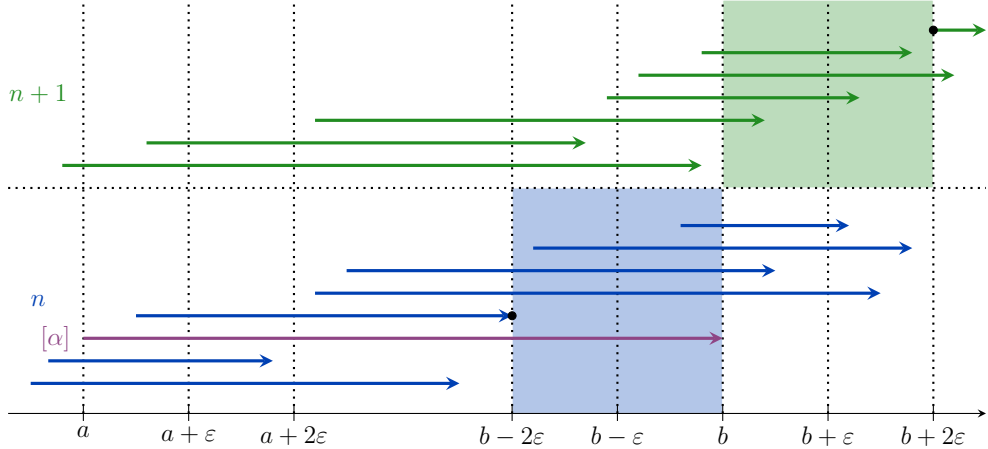

\section{Related work}

After the publication of the first version of this paper we were made aware that a precise relationship with \cite{EdelsMoro} and \cite{Morozov} would be helpful. In these works the authors consider what in our terminology would be phrased as a rigidity of persistence pairs (birth simplex, death simplex) with respect to transpositions of adjacent simplices. On the other hand, we consider rigidity of the terminal simplex of a homology class or a bar with respect of $\varepsilon$-perturbations of filtration functions. The concepts are different insofar as the first one depends only on the permutation of simplices, treats single transpositions of simplices and looks to preserve pairing in persistent homology; on the other hand, our concept treats only the terminal simplices of a homology class, and considers all homology classes (not just the ones generating bars) and all $\varepsilon$-perturbations. The mentioned concepts of rigidity are different, as is demonstrated by Figure \ref{FigDifference}.

The approach of \cite{EdelsMoro} and \cite{Morozov} is algorithmic, has been used to demonstrate change in persistence diagrams via one-parameter modification of filtration functions and yields a proof of stability theorem. The technical treatment is based on the analysis of the matrix reduction based persistence algorithm.

Our approach is to look more generally at any single homology class and consider all $\varepsilon$-perturbations within our direct treatment of homology classes. Along the way we describe how $\varepsilon$-perturbations affect permutations of simplices (Subsection \ref{SubsFM}), at what values of $\varepsilon$ do terminal simplices potentially change and how (Proposition \ref{PropStep2}), and demonstrate that their number might increase by more than one despite a ``single new transposition'' in an incremental increase of $\varepsilon$ (Example \ref{example6}). 

Different rigidity concepts understandably generate different results. Sufficient conditions for rigidity of pairings are given by Nested-Disjoint Lemma in \cite{Morozov}, while sufficient conditions for rigidity of terminal simplices in barcodes are presented in Theorem \ref{ThmBarcodeFinal}. Note that example in Figure \ref{FigMainIntro1g} satisfies only the conditions of Theorem \ref{ThmBarcodeFinal}. We do believe though that the treatment of \cite{Morozov} could be expanded and combined with our results on $\varepsilon$-perturbations to yield another proof of Theorem \ref{ThmBarcodeFinal}. On the other hand, such treatment would not suffice for our main and most general result: Theorem \ref{ThmConseq1}.

\begin{figure}[htbp]
\begin{center}
\includegraphics[scale=1]{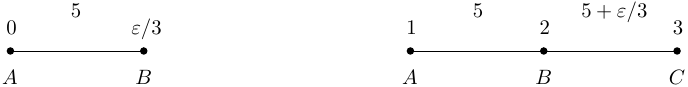}
\caption{Two simplicial complexes with the labels above the simplices indicating filtration values. Let $\varepsilon \in (0,1)$.
The only persistence pair in the left filtration is $(A,AB)$ while an $\varepsilon$-perturbation may change it to $(B,AB)$. On the other hand, the only homology class that is terminated is $[A]-[B]$ and its terminal simplex is always $AB$, hence $[A]-[B]$ is $\varepsilon$-rigid. On the right we see a filtration whose two persistence pairs are rigid with respect to $\varepsilon$-perturbations despite the terminal simplex of $[A]-[C]$ changing from $AB$ to $AC$. The concepts of rigidity of persistence pairs and rigidity of the terminal simplex of a homology class thus differ.}
\label{FigDifference}
\end{center}
\end{figure}

A different treatment of instability of information provided by persistent homology is given in \cite{Bub}. There, the authors recast an unstable output of persistent homology as a real valued function and average it over small perturbations to obtain a stable output. For example, while the cycle generating a persistent homology class is unstable, the approach of \cite{Bub} yields, roughly speaking, a distribution of generating cycles over small perturbations. As such, this approach is aimed at stabilizing the potentially unstable outputs of persistent homology. On the other hand, our work aims to detect the instability of terminal simplices from the structure of filtration or persistence diagram. 

\section{Conclusions and further work}

In this paper we have established sufficient conditions for the rigidity of terminal simplices. One of the main advantages of our result is that the conditions only depend on the persistence diagram. In our subsequent work we intend to extend our approach to a more geometric setting, treating persistence diagrams arising from the popular Vietoris-Rips filtrations on metric spaces. In this setting the instability of terminal simplices should be measured by the distance between terminal simplices in the metric space, as opposed to ``combinatorial'' proximity of this paper. In particular we plan to explore the instability in case of persistence diagrams arising from Vietoris-Rips filtrations of geodesic spaces. Recent results on $S^1$ \cite{AA, Moy} indicate that in this case, the terminal simplices of $1$-dimensional homology form a $1$-parameter family which consequently generates a $3$-dimensional homology class. We intend to provide a general treatment of this phenomenon.



\begin{thebibliography}{99}

\bibitem{AA}
M. Adamaszek and H. Adams,
\emph{The Vietoris-Rips complexes of a circle}. 	
Pacific Journal of Mathematics 290 (2017), 1--40.

\bibitem{Ad5}
M. Adamaszek, H. Adams, and S. Reddy,
\emph{On Vietoris-Rips complexes of ellipses}, 
Journal of Topology and Analysis 11 (2019), 661-690.

\bibitem{ACos}
H. Adams and B. Coskunuzer,
\emph{Geometric Approaches on Persistent Homology}, 
arXiv:2103.06408.

\bibitem{Bauer} U.~Bauer and M.~Lesnick, \emph{Induced Matchings and the Algebraic Stability of Persistence Barcodes}, Journal of Computational Geometry 6:2 (2015), 162--191.

\bibitem{Bub}
P. Bendich, P. Bubenik, and A. Wagner,  \emph{Stabilizing the unstable output of persistent homology computations}, J Appl. and Comput. Topology 4, 309--338 (2020). 

\bibitem{Bubenik2020}
P. Bubenik, M. Hull, D. Patel, and B. Whittle, 
\emph{Persistent homology detects curvature}, 
Inverse Problems 36(2), 2020.

\bibitem{Chao} Chao Chen and Daniel Freedman, \emph{Hardness results for homology localization}, in Proceedings of the Twenty-First Annual ACM-SIAM Symposium on Discrete Algorithms (SODA), 2010.

\bibitem{EdelsMoro}D.~Cohen-Steiner, H.~Edelsbrunner and D.~Morozov, \emph{Vines and vineyards by updating persistence in linear time}, Proc. 22nd Ann. Sympos. Comput. Geom. (2006), 119--126.

\bibitem{EdelsHarer}
H.~Edelsbrunner and J.~Harer \emph{Computational Topology - an Introduction}, American Mathematical Society (2010).

\bibitem{EdelsZomo} H.~Edelsbrunner, D.~Letscher and A.~Zomorodian, \emph{Topological Persistence and Simplification}, Discrete Comput. Geom. {\bf 28} (2002), 511--533.

\bibitem{Haus}Jean-Claude Hausmann, \emph{On the Vietoris-Rips complexes and a cohomology theory for metric spaces.} Annals of Mathematics Studies, 138:175--188, 1995.

\bibitem{Lat}J. Latschev, \emph{Vietoris-Rips complexes of metric spaces near a closed Riemannian manifold.} Archiv der Mathematik, 77(6):522--528, 2001.

\bibitem{Memoli} S. Lim, F. M\' emoli, and O.B. Okutan, 
\emph{Vietoris-Rips persistent homology, injective
metric spaces, and the filling radius}, arXiv:2001.07588, 2020.

\bibitem{Morozov}
D. Morozov,
\emph{Homological illusions of persistence and stability}, Ph.D. Dissertation. Duke University, USA, 2008. 

\bibitem{Moy}
M. Moy, 
\emph{Vietoris-Rips Metric Thickenings of the Circle}, 
arXiv:2206.03539.

\bibitem{ZVCont} 
\v Z. Virk, 
\emph{Contractions in persistence and metric graphs}, 
Bull. Malays. Math. Sci. Soc. 45 (2022), 2003--2016.

\bibitem{ZV} 
\v Z. Virk, 
\emph{1-Dimensional Intrinsic Persistence of geodesic spaces}, 
Journal of Topology and Analysis 12 (2020), 169--207.
 
 \bibitem{ZVCounterex}
\v Z. Virk,
\emph{A Counter-Example to Hausmann's Conjecture}, Foundations of Computational Mathematics (2021). 

\bibitem{ZV2}
\v Z. Virk, 
\emph{Footprints of geodesics in persistent   homology}, 
Mediterranean Journal of Mathematics 19 (2022).

\bibitem{ZV3}
\v Z. Virk, 
\emph{Rips complexes as nerves and a Functorial Dowker-Nerve Diagram}, Mediterr. J. Math. 18 (2021). 


\end{thebibliography}
\end{document}